\documentclass[11pt]{article}
\usepackage{amsmath,amsfonts,amssymb,amsthm,enumerate,graphicx,xcolor,url}
\usepackage{mathtools}
\usepackage{tikz,authblk}
\usepackage{subfig}
\DeclarePairedDelimiter{\ceil}{\lceil}{\rceil}
\DeclarePairedDelimiter{\floor}{\lfloor}{\rfloor}
\newtheorem{theorem} {{\textsf{Theorem}}}[section]
\newtheorem{proposition}[theorem]{{\textsf{Proposition}}}
\newtheorem{corollary}[theorem]{{\textsf{Corollary}}}

\newtheorem{definition}[theorem]{{\textsf{Definition}}}
\newtheorem{remark}[theorem]{{\textsf{Remark}}}

\newtheorem{lemma}[theorem]{{\textsf{Lemma}}}
\newcommand{\Star}{\mbox{\upshape st}\,}
\newcommand{\lk}{\mbox{\upshape lk}\,}

\begin{document}

\title{A characterization of normal 3-pseudomanifolds with at most two singularities}
\author{Biplab Basak, Raju Kumar Gupta and Sourav Sarkar}

\date{}

\maketitle

\vspace{-15mm}
\begin{center}

\noindent {\small Department of Mathematics, Indian Institute of Technology Delhi, New Delhi 110016, India$^{1}$.}

\end{center}

\footnotetext[1]{{\em E-mail addresses:} \url{biplab@iitd.ac.in} (B.
Basak), \url{Raju.Kumar.Gupta@maths.iitd.ac.in} (R. K. Gupta), \url{Sourav.Sarkar@maths.iitd.ac.in} (S. Sarkar).}

\medskip

\begin{center}
\date{June 22, 2023}
\end{center}

\hrule

\begin{abstract}
Characterizing face-number-related invariants of a given class of simplicial complexes has been a central topic in combinatorial topology. In this regard, one of the well-known invariants is $g_2$. Let $K$ be a normal $3$-pseudomanifold  such that  $g_2(K) \leq g_2(\lk (v)) + 9$ for some vertex $v$ in $K$. Suppose either $K$ has only one singularity or $K$ has two singularities (at least) one of which is an $\mathbb{RP}^2$-singularity. We prove that $K$ is obtained from some boundary complexes of $4$-simplices by a sequence of operations of types connected sums,  bistellar $1$-moves, edge contractions, edge expansions, vertex foldings, and edge foldings. In  case $K$ has one singularity, $|K|$ is a  handlebody with its boundary coned off. Further, we prove that the above upper bound is sharp for such normal $3$-pseudomanifolds.
\end{abstract}

\noindent {\small {\em MSC 2020\,:} Primary 05E45; Secondary 52B05, 57Q05, 57Q25, 57Q15.

}

\medskip

\section{Introduction}
For a $d$-dimensional finite simplicial complex $K$, $g_2(K)$ is defined by  $g_2:=f_1-(d+1)f_0 + \binom{d+2}{2}$, where $f_0$ and $f_1$ denote the number of vertices and edges in $K$. The study on $g_2$ has been illuminated in a different prospect due to the lower bound conjecture for 3- and 4-manifolds by Walkup \cite{Walkup} in 1970. He proved that for any closed and connected triangulated $3$-manifold $K$, $g_2(K)\geq 0$, and the equality occurs if and only if $K$ is a triangulation of a stacked sphere. Barnette \cite{Barnette1, Barnette2, Barnette3} proved that if $K$ is the boundary complex of a simplicial $(d+1)$-polytope or, more generally, a triangulation of a connected $d$-manifold, then $g_2(K)\geq 0$.
%In 1987, Kalai \cite{Kalai} proved that if $K$ is a normal pseudomanifold of dimension at least 3 with the 2-dimensional links as triangulated spheres, then for any face $\sigma$ of co-dimension at least 3, the inequality  $g_2(K)\geq g_2(\lk (\sigma))$ holds, and therefore the non-negativity of $g_2$ holds.  Fogelsanger's  result in \cite[Chapter 8]{Fogelsanger} implies that the result is true for any  normal $d$-pseudomanifold. In \cite{Gromov}, Gromov has similar work on the non-negativity of $g_2$.
In \cite{Kalai}, Kalai proved that if $K$ is a normal pseudomanifold of dimension at least 3 with the 2-dimensional links as triangulated spheres, then $g_2(K)\geq g_2(\lk (\sigma))$  for every face $\sigma$ of co-dimension at least 3. Fogelsanger's  results in \cite[Chapter 8]{Fogelsanger} made it possible to remove the restriction on 2-dimensional links. Therefore,  $g_2(K) \geq 0$  for every normal $d$-pseudomanifold $K$.  In \cite{Gromov}, Gromov has similar work on the non-negativity of $g_2$.

Based on values of $g_{2}$, several classifications of combinatorial manifolds and normal pseudomanifolds are studied in the literature.  In \cite{Swartz2008}, Swartz proved that the number of combinatorial manifolds, up to PL-homeomorphism, of a given dimension $d$ with an upper bound on $g_2$, is finite. The combinatorial characterizations of  normal $d$-pseudomanifolds are known due to Kalai \cite{Kalai} (for $g_2 = 0$), Nevo and Novinsky  \cite{NevoNovinsky} (for $g_2 = 1$), and Zheng \cite{Zheng} (for $g_2 = 2$). In all three cases, the normal pseudomanifold is the boundary complex of a simplicial polytope. The classification of all triangulated pseudomanifolds of dimension $d$ with at most $d + 4$ vertices can be found in \cite{BagchiDatta98}. For further developments in this direction, one may refer to \cite{NovikSwartz, Swartz2009, TayWhiteWhiteley}.

 In  \cite{BasakSwartz}, Basak and Swartz introduced two new concepts, viz., vertex folding and edge folding. For a normal 3-pseudomanifold $K$ with exactly one singularity, they proved that if $g_2(K) = g_2(\lk (v))$ for some vertex $v$ in $K$, then $|K|$ is a  handlebody with its boundary coned off. This leads to a natural question: what will be the maximum value of $n\in\mathbb{N}$, for which $g_2{(K)} \leq g_2(\lk (t)) + n$ implies that $|K|$ is a  handlebody with its boundary coned off, where $t$ is a vertex of $K$ with $g_2(\lk (t)) \geq g_2(\lk (v))$ for any other vertex $v$? In this article, we answer this question. We prove that if  $g_2{(K)} \leq g_2(\lk (t)) + 9$, then $|K|$ is a  handlebody with its boundary coned off. Moreover, the above upper bound is sharp for such normal $3$-pseudomanifolds (cf. Corollary \ref{corollary:handlebody}). We give a combinatorial characterization of normal 3-pseudomanifolds with at most two singularities, where, in the case of two singularities, one is assumed to be an $\mathbb{RP}^2$-singularity. If $K$ has no singular vertices, then from \cite{BasakGupta, Walkup}, we know that  $g_2{(K)} \leq g_2(\lk (t)) + 9$ implies $K$ is a triangulated $3$-sphere and is obtained from some boundary complexes  of $4$-simplices by a sequence of operations of types connected sums,  bistellar $1$-moves, edge contractions, and edge expansions. In this article, we extend this characterization to normal 3-pseudomanifolds with at most two singularities, that reads as follows:

\begin{theorem}\label{main theorem}
Suppose $g_2(K) \leq g_2(\lk (v)) + 9$ for some vertex $v$ of a normal $3$-pseudomanifold $K$. If either $K$ has only one singularity or $K$ has two singularities (at least) one of which is an $\mathbb{RP}^2$-singularity, then $K$ can be obtained from some boundary complexes  of $4$-simplices by a sequence of operations of types connected sums,  bistellar $1$-moves, edge contractions, edge expansions, vertex foldings, and edge foldings. Further, the above upper bound is sharp for such normal $3$-pseudomanifolds.
\end{theorem}

\section{Preliminaries}
A {\em simplicial complex} $K$ is a finite collection of simplices in $\mathbb{R}^m$ for some $m\in\mathbb{N}$,  such that for any simplex $\sigma\in K$, all of its faces are in $K$, and for any two simplices $\sigma , \tau\in K$, $\sigma\cap \tau$ is either empty or a face of both.
We assume that the empty set $\emptyset$ (is a simplex of dimension $-1$) is a member of every simplicial complex. We define the dimension of a simplicial complex $K$ to be the maximum of the dimension of simplices in $K$. For a $d$-dimensional simplicial complex $K$, the $f$-vector is defined  as  a $(d+2)$-tuple $(f_{-1},f_0,\dots,f_d)$, where $f_{-1}=1$ and for $0\leq i\leq d$, $f_i$ denotes the number of $i$-dimensional faces in $K$. A maximal face in a simplicial complex $K$ is called a {\em facet}, and if all the facets are of the same dimension, then  $K$ is said to be a {\em pure} simplicial complex. A {\em subcomplex} of $K$ is a simplicial complex $T\subseteq K$.  Let $S\subset V(K)$, where $V(K)$ is the vertex set of $K$. Then the subcomplex of $K$ induced on the vertex set $S$ is denoted by $K[S]$.  By $|K|$ we mean the union of all simplices in $K$ together with the subspace topology induced from $\mathbb{R}^m$.  A {\em triangulation} of a polyhedra $X$ is a simplicial complex $K$ together with a PL-homeomorphism between $|K|$ and $X$.

 Two simplices $\sigma = u_0u_1\cdots u_k$ and $\tau = v_0v_1\cdots v_l$ in $\mathbb{R}^n$ for some $n\in \mathbb{N}$ are called skew if $u_0,\dots ,u_k,v_0,\dots,v_l$ are affinely independent. In that case $u_0\cdots u_kv_0\cdots v_l$ is a $(k + l + 1)$-simplex and is denoted by  $\sigma\star\tau$ or $\sigma\tau$. Two simplicial complexes $K$ and $L$ in some $\mathbb{R}^n$ are called skew if $\sigma$ and $\tau$ are skew for all $\sigma\in K$ and $\tau\in L$. If $K$ and $L$ are skew then we define $K\star L = K \cup L \cup\{\sigma\tau : \sigma\in K, \tau\in L\}$. The simplicial complex $K\star L$ is called the join of $K$ and $L$. If $K$ and $L$ are two simplicial complexes in $\mathbb{R}^p$ and $\mathbb{R}^q$ respectively, then we can define their join in a bigger space. More explicitly, let $i_1 : \mathbb{R}^p\to \mathbb{R}^{p+q+1}$, $i_2 : \mathbb{R}^q\to \mathbb{R}^{p+q+1}$ be the inclusion maps given by $i_1(x_1,\dots, x_p) = (x_1,\dots, x_p,0,\dots, 0)$ and $i_2(x_1,\dots, x_q) = (0,\dots,0,x_1,\dots, x_q,1)$. Let $K' := \{i_1(\sigma) : \sigma\in K\}$ and $L' := \{i_2(\tau : \tau\in L\}$. Then $K\cong K', L\cong L'$ and $K'$ and $L'$ are skew in $\mathbb{R}^{p+q+1}$. We define $K\star L=K'\star L'$.
% 
% If $K$ is a simplicial complex in $\mathbb{R}^p$ and $L$ is a simplicial complex in $\mathbb{R}^q$, for some $p, q\in \mathbb{N}$, then $K$ and $L$ are skew. Then the join $K \star L$ of $K$ and $L$ in $\mathbb{R}^{p+q+1}$ is defined by $\{\alpha\leq \sigma\tau : \sigma\in K, \tau\in L\}$.
% 
%By {\em join} of two simplices $\sigma$ and $\tau$ of dimensions $i,j$ respectively we mean the simplex $\{\lambda a + \mu b : a\in\sigma, b\in\tau; \lambda, \mu\in [0,1]$ and $\lambda+\mu=1\}$ and denote it by  $\sigma\star\tau$ or $\sigma\tau$. Two simplicial complexes $K_1$ and $K_2$ are independent if $\sigma\tau$ is an $(i+j+1)$-simplex for each $i$-simplex $\sigma\in K_1$ and $j$-simplex $\tau\in K_2$. The join of two independent simplicial complexes $K_1$ and $K_2$ is defined as $K_1\cup K_2\cup\{\sigma\tau : \sigma\in K_1, \tau\in K_2\}$ and is denoted by $K_1\star K_2$.
For a simplex $\sigma$ in $\mathbb{R}^p$  and a simplicial complex $K$ in $\mathbb{R}^q$, by $\sigma \star K$ we mean $\{\alpha\,:\,\alpha \leq \sigma \}\star K$ in $\mathbb{R}^{p+q+1}$ for some $p, q\in \mathbb{N}$. The {\em link} of a face $\sigma$ in $K$ is defined as $\{ \gamma\in K : \gamma\cap\sigma=\emptyset$ and $ \gamma\sigma\in K\}$, and it is denoted by $\lk (\sigma,K)$. The {\em star} of a face $\sigma$ in $K$  is defined as $\{\alpha : \alpha\leq\sigma \beta; \beta\in \lk (\sigma,K)\}$, and it is denoted by $\Star (\sigma,K)$. If the underlying simplicial complex is specified, then we use the notations  $\lk (\sigma)$ and $\Star (\sigma)$ to refer to the link and star of the face $\sigma$, respectively. For every face $\sigma$ in $K$, by $d(\sigma,K)$ (or, $d(\sigma)$ if $K$ is specified) we mean the number of vertices in $\lk (\sigma)$. For two vertices $x$ and $y$, $(x,y]$ denotes the semi-open and semi-closed edge $xy$, where $y\in (x,y]$ but $x \not \in (x,y]$. By $(x,y)$, we denote the open edge $xy$, where $x, y\not \in (x,y)$.  By $B_{x_1,\dots,x_m}(p;q)$, we denote the bi-pyramid with $m$ base vertices $x_1,\dots,x_m$ and apexes $p$ and $q$.

\par A {\em normal $d$-pseudomanifold} without boundary (respectively with boundary) is a connected pure simplicial complex in which every face of dimension $(d-1)$ is contained in exactly two (respectively at most two) facets and the links of all the simplices of dimension $\leq (d-2)$ are connected. For a normal  $d$-pseudomanifold $K$ with a connected boundary, its boundary $\partial K$ is a normal $(d-1)$-pseudomanifold whose facets are $(d-1)$-dimensional faces of $K$, each of which is contained in exactly one facet of $K$. Throughout the article, by a normal $d$-pseudomanifold, we mean a normal $d$-pseudomanifold without boundary. For a simplex $\sigma$, its boundary complex $\partial (\sigma)$ is the collection of all of its proper faces.
Let $K$ be a normal $d$-pseudomanifold with boundary $\partial K$. Let $K'=K\cup (t\star \partial K)$, where $t$ is a new vertex. If $\partial K$ is connected, then $K'$ is a normal $d$-pseudomanifold without boundary. We say $K'$ is obtained from $K$
by coning off the boundary at $t$, and the topological space $|K'|$ is the  topological space $|K|$ with its boundary coned off. In a normal $d$-pseudomanifold $K$, the vertices whose links are triangulated spheres are called {\em non-singular} vertices, and the remaining are called {\em singular} vertices. Let $v$ be a singular vertex in a normal $d$-pseudomanifold $K$ with $|\lk (v,K)|\cong S$. In this case, we say $K$ has an $S$-singularity at $v$.

\begin{definition}\label{Edge contraction}{\rm Let $K$ be a normal $d$-pseudomanifold and $u,v$ be two vertices in $K$ such that $uv \in K$, and $\lk (u,K)\cap \lk (v,K)=\lk (uv,K)$. Consider $K'= K \setminus (\{\alpha\in K : u\leq \alpha\} \cup \{\beta\in K : v\leq \beta\})$ and  $K_1=K'\cup (w \star \partial K')$, where $w$ is a new vertex. Then, we say $K_1$ is obtained from $K$ by an {\em edge contraction} at $uv$.

Let $L$ be a normal $d$-pseudomanifold and $w$  be a vertex in $L$. Let $S$ be an induced normal $(d-2)$-pseudomanifold in $\lk (w)$ such that $S$ separates $\lk (w)$ into two portions, say $L_1$ and $L_2$, where $L_1$ and $L_2$ are normal $(d-1)$-pseudomanifolds with the same boundary complex $S$. Consider $L'= (L \setminus \{\alpha\in L : w\leq\alpha\}) \cup (u \star L_1)\cup (v\star L_2) \cup (uv \star S)$, where $u$ and $v$ are two new vertices. We say $L'$ is obtained from $L$ by an {\em edge expansion}. Note that, $L$ can be obtained from $L'$ by contracting the edge $uv$.}
\end{definition}

\begin{definition}\label{definition:operation}{\rm
Let $K$ be a normal $3$-pseudomanifold. 
\begin{enumerate}
%\item[$(A)$]  Let $w$  be a vertex in $K$. Let $S$ be an induced cycle in $\lk (w)$ such that $S$ separates $\lk (w)$ into two portions, say $D$ and $R$, where $D$ is a triangulated disc. Further, $D$ and $R$ have the same boundary complex $S$. Let $K'= (K \setminus \{\alpha\in K : w\leq\alpha\}) \cup (u \star D)\cup (v\star R) \cup (uv \star S)$, where $u$ and $v$ are two new vertices. We say $K'$ is obtained from $K$ by an edge expansion. Note that, $K$ is obtained from $K'$ by contracting the edge $uv$. Thus, an edge expansion is a reverse operation of an edge contraction.
%

%\item[$(A')$] Let $K$ be a normal 3-pseudomanifold. Let $w\in K$ and $C$ be an $n$-cycle in $\lk (w)$ such that $C$ separates $\lk (w)$ into two portions, say $D$ and $E$, where $D$ is a disc and $C$ is the common boundary of $D$ and $E$. Let $K'= (K \setminus \{\alpha\in K : w\in \alpha\}) \cup (uv\star C_n) \cup (u\star D)\cup (v\star E)$. We say $K'$ is obtained from $K$ by an {\em edge expansion} at $w$.  Further, $g_2(K')=g_2(K)+n-3$ and $|K|'\cong |K|$.

\item[$(A)$] Let $uv$ be an edge in $K$ such that $\lk (uv)=\partial(abc)$ and $abc \not \in K$. Consider $K'=(K\setminus\{\alpha\in K: uv \leq \alpha\})\cup\{abc, uabc, vabc\}$. Since $abc\not \in K$, we have $|K'|\cong |K|$. Further, $g_2(K')=g_2(K)-1$. We say $K'$ is obtained from $K$ by a {\em bistellar $2$-move}. The reverse operation is called a {\em bistellar $1$-move}.

\item[$(B)$]  Let $w$  be a non-singular vertex in $K$ such that $\partial (abc) \subset \lk (w)$ , where $abc\not \in K$. Then $\partial (abc)$ separates $\lk (w)$ into two portions, say $D_1$ and $D_2$, where $D_1$ and $D_2$ are triangulated discs with the same boundary complex $\partial (abc)$. Consider $K'= (K \setminus \{\alpha\in K : w\leq\alpha\}) \cup (u\star D_1)\cup (v\star D_2) \cup (abc \star \partial(uv))$, where $u$ and $v$ are two new vertices. Then $g_2(K')=g_2(K)-1$ and $|K|'\cong |K|$.  Note that this combinatorial operation  is a combination of an edge expansion and a bistellar $2$-move.

%\item [$(C)$] Consider a combinatorial operation as a combination of an edge expansion and an edge contraction. This combinatorial operation is the reverse of itself.
\end{enumerate}
}
\end{definition}

\begin{remark} \label{remark:operation}{\rm
Let $K$ be a normal $3$-pseudomanifold. Let $K'$ be obtained from $K$ by one of the following combinatorial operations: $(i)$ an edge contraction, $(ii)$ a bistellar $2$-move, $(iii)$ the operation as in Definition \ref{definition:operation} $(B)$, or $(iv)$ a combination of an edge expansion and an edge contraction. Then, $K$ is obtained from $K'$ by one of the following combinatorial operations: $(i)$ an edge expansion, $(ii)$ a bistellar $1$-move, $(iii)$ a combination of  a bistellar $1$-move and an edge contraction, or $(iv)$ a combination of an edge expansion and an edge contraction.}
\end{remark}

\begin{lemma}\label{lemma:homeomorphic}
For $d\geq 3$, let $K$ be a normal $d$-pseudomanifold. Let $uv$ be an edge in $K$ such that $\lk (u,K)\cap \lk (v,K)= \lk (uv,K)$ and $|\lk (v,K)| \cong \mathbb{S}^{d-1}$. If $K_1$ is the normal pseudomanifold obtained from $K$ by contracting the edge $uv$, then  $|K|\cong |K_1|$.
\end{lemma}

\begin{proof}
Since $\lk (u,K)\cap \lk (v,K)=\lk (uv,K)$, the edge contraction is possible. Let $w$ be the new vertex in $K_1$ obtained by identifying the vertices $u$ and $v$ in $K$. Let $K'=K \setminus (\{\alpha\in K : u \leq \alpha\} \cup \{\beta\in K : v \leq \beta\})$. Then $K'$ is a normal $d$-pseudomanifold with boundary and $\partial K'=\partial(\Star (u,K)\cup \Star (v,K))$. Since $\lk (u,K)\cap \lk (v,K)=\lk (uv,K)$ and $|\lk (v,K)| \cong \mathbb{S}^{d-1}$, $\lk (v,K)\setminus \{\alpha\in \lk (v,K) : u \leq \alpha\}$ is a triangulated $(d-1)$-ball, say $D$, with boundary $\lk (uv,K)$. Further, $K'\cap \Star (v,K)=D$. Since $|\lk (v,K)| \cong \mathbb{S}^{d-1}$,
  $|\Star (v,K)| \cong \mathbb{D}^d$. Therefore, $|K'|$ is PL-homeomorphic to $|K'\cup \Star (v,K)|$.  Let $K'' := K'\cup \Star (v,K)$. Then $K = K'' \cup (u\star \partial K'')$ and $K_1 = K' \cup (w\star \partial K')$. Since $|K''|$ and $|K'|$ are PL-homeomorphic, $|K|$ and $|K_1|$ are also PL-homeomorphic.
\end{proof}

From \cite[Lemma 10.8]{Walkup}, we have the following result (see \cite[Section 4]{BasakGupta} for more details):

\begin{proposition}[\cite{BasakGupta, Walkup}] \label{proposition:BG}
	If $K$ is a triangulated $3$-manifold with $g_2(K)\leq 9$, then $K$ is a triangulated $3$-sphere, and is obtained from some boundary complexes  of $4$-simplices by a sequence of operations of types connected sums,  bistellar $1$-moves, edge contractions, and edge expansions. 
\end{proposition}

 Let $K$ be a normal $d$-pseudomanifold and $f:V(K) \to \mathbb{R}^{d+1}$ be a function.  A {\em  stress} of $f$ is a function $\omega:E(K) \to \mathbb{R}$ such that for every vertex $v \in V,$
$\sum_{vu \in E(K)} \omega(vu) (f(v) -f(u)) = 0$, where $E(K)$ denotes the set of edges in $K$.
The set of all stresses of $f$ is an $\mathbb{R}$-vector space, which we denote by $\mathcal{S}(K_f).$ From  \cite{Kalai}, we know that if $K$ is a  normal $d$-pseudomanifold (or a cone over a normal $(d-1)$-pseudomanifold) and $f$ is a generic map from $V$ to $\mathbb{R}^{d+1},$  then $\dim \mathcal{S}(K_f) = g_2(K).$ Using the fact that a stress $\omega$ on a subcomplex of $K$ can be extended to $K$ by setting $\omega(uv)=0$ for any $uv$ not in the subcomplex, we have the following result (due to Kalai \cite{Kalai}):

\begin{lemma}[\cite{BasakSwartz, Kalai}] \label{g2>face}
Let $K$ be a normal $3$-pseudomanifold and $v$ be a vertex in $K$. Then $g_2(K) \geq g_2(\Star (v,K ))=g_2(\lk (v,K)).$ Moreover, if  $u$ and $v$ are two vertices  in $K$ such that $\Star (u,K) \cap \Star (v,K) =\emptyset$, then $g_2(K) \geq g_2(\Star (u,K ))+ g_2(\Star (v,K ))=g_2(\lk (u,K ))+ g_2(\lk (v,K )).$ 
\end{lemma}

Let $K$ be a pure simplicial complex and $\sigma_1, \sigma_2$ be two facets of $K$. A bijection $\psi: \sigma_1\to \sigma_2$ is said to be {\em admissible} (cf. \cite{BagchiDatta}) if for any vertex $x\leq \sigma_1$, length of every path between $x$ and $\psi(x)$ is at least $3$. In this contest, any bijective map between two facets from different connected components of $K$ is admissible. Now if $\psi$ is an admissible bijection between $\sigma_1$ and $\sigma_2$, by identifying all the faces $\rho_1\leq \sigma_1$ with $\psi(\rho_1)$ and removing the identified facets, we get a new complex, say $K^{\psi}$. If $\sigma_1$ and $\sigma_2$ are from the same component of $K$, then we say $K^{\psi}$ is formed via a {\em handle addition} (cf. \cite{BasakSwartz}) to $K$. If $\sigma_1$ and $\sigma_2$ are from different components of $K$,  then we say $K^{\psi}$ is formed via a {\em connected sum} (cf. \cite{BasakSwartz}), and we write it as $K^{\psi}= K_1 ~\#_\psi~ K_2$, where $\sigma_1\in K_1$ and $\sigma_2\in K_2$.
 If $\bar{x}$ represents the identified vertices $x$ and $\psi(x)$ in the connected sum or handle addition, then $\lk (\bar x,K^{\psi})=\lk (x,K_1) ~\#_{\psi} ~ \lk (\psi(x),K_2)$, and for all other vertices, the links will remain the same.
 
 A {\em missing triangle} of $K$ is a triangle $\sigma$ such that $\sigma \notin K$ but $\partial(\sigma)\subset K$. Similarly, a {\em missing tetrahedron} of $K$ is a tetrahedron $\tau$ such that $\tau\notin K$ but $\partial(\tau)\subset K$.

\begin{lemma}[\cite{BasakSwartz}] \label{connected sum and handle addition}
Let $K$ be a normal $3$-dimensional pseudomanifold, and suppose $\tau$ is a missing tetrahedron in $K.$  If for every vertex $x \leq \tau$,  the missing triangle formed by the other three vertices separates the link of $x,$  then $K$ is formed using either a handle addition or a connected sum.  
\end{lemma}

A straightforward computation shows that for a $d$-dimensional complex $K$, a handle addition and a connected sum satisfy the following:
\begin{equation} \label{g_2: handles}
g_2(K^\psi) = g_2(K) + \binom{d+2}{2},
\end{equation}

\begin{equation} \label{g_2:connected sum}
g_2(K_1 ~\#_\psi~ K_2) = g_2(K_1) + g_2(K_2).
\end{equation}

\begin{lemma} \label{connected sum}
Let $K$ be a normal  $3$-pseudomanifold such that $g_2{(K)} \leq g_2(\lk (v,K)) + 9$ for some vertex $v$. Let $\sigma$ be a missing tetrahedron in $K$ such that for every vertex $x \leq \sigma$,  the missing triangle formed by the other three vertices separates the link of $x$.  Then $K$ is formed using a connected sum.  
\end{lemma}

\begin{proof}
By Lemma \ref{connected sum and handle addition}, $K$ is formed using either a handle addition or a connected sum.  If possible, let $K$ be formed using a handle addition from $K'$ through the admissible bijection $\psi: \sigma_1\to \sigma_2$. Then the identified simplex $\sigma$ (obtained by identifying $\psi(\sigma_1)$ with $\sigma_2$) is a missing tetrahedron. If $v\not \leq \sigma$, then $g_2(\lk (v,K))=g_2(\lk (v,K')) \leq g_2(K') = g_2(K)-10.$ This is a contradiction. If possible, let $v \leq \sigma$ be obtained by identifying  $v_1 \leq \sigma_1$ and  $v_2 \leq \sigma_2$ in $K'$. Then  $g_2(\lk (v,K))=g_2(\lk (v_1,K'))  + g_2(\lk (v_2,K')).$
Since $\psi$ is admissible, $\Star (v_1,K')$ and $\Star (v_2,K')$ are disjoint. Then it follows from Lemma \ref{g2>face} that $g_2(K')\geq g_2(\lk (v_1,K'))  + g_2(\lk (v_2,K'))$. Therefore, $g_2(K')\geq g_2(\lk (v,K))$ and $g_2(K)=g_2(K')+10$. This implies, $g_2(K)\geq g_2(\lk (v,K))+10$, which is a contradiction. Therefore, $K$ is formed using a connected sum.  
\end{proof}

Handle addition and connected sum are standard parts of combinatorial topology, but the operation of {\em  folding} was recently introduced in \cite{BasakSwartz}.

\begin{definition}[Vertex folding \cite{BasakSwartz}] 
{\rm 
Let $\sigma_1$ and $\sigma_2$ be two facets of a simplicial complex $K$, whose intersection is a single vertex $x.$  A bijection $\psi:\sigma_1 \to \sigma_2$ is {\em  vertex folding admissible} if $\psi(x) = x$ and for all other vertices $y$ of $\sigma_1$, the only path of length two from $y$ to $\psi(y)$ is $P(y, x, \psi(y)).$ For a vertex folding admissible map $\psi$, we can form the complex $K^\psi_x$ by identifying all faces $\rho_1 \leq \sigma_1 $ and $\rho_2 \leq \sigma_2 $, such that $\psi(\rho_1) = \rho_2,$ and then removing the facet formed by identifying $\sigma_1$ and $\sigma_2.$  In this case, we say that $K^\psi_x$ is a {\em  vertex folding} of $K$  at $x.$   In a similar spirit, $K$ is a {\em  vertex unfolding} of $K^\psi_x.$ }
\end{definition}

\noindent A straightforward computation shows that if $K^\psi_x$ is obtained from a $d$-dimensional  simplicial complex  $K$ by a vertex folding at $x$, then

 \begin{equation} \label{folding g2}
g_2(K^\psi_x) = g_2(K)+\binom{d+1}{2}.
\end{equation}
The definition of {\em edge folding} follows the same pattern as vertex folding.
\begin{definition}[Edge folding \cite{BasakSwartz}]
{\rm 
Let $\sigma_1$ and $\sigma_2$ be two facets of a simplicial complex $K$, whose intersection is an edge $uv$. A bijection $\psi : \sigma_1 \to \sigma_2$ is {\em edge folding admissible} if $\psi(u)= u, \psi(v) = v$, and for all other vertices $y$ of $\sigma_1$, all paths of length two or less from $y$ to $\psi(y)$ pass through either $u$ or $v$. Identify all faces $\rho_1\leq \sigma_1$ and $\rho_2 \leq \sigma_2$, such that $\psi : \rho_1\to \rho_2$ is a bijection. The complex obtained by removing the facet resulting from identifying $\sigma_1$ and $\sigma_2$ is denoted by $K^\psi_{uv}$ and is called an {\em edge folding} of $K$ at $uv$. In a similar spirit,  $K$ is an {\em edge unfolding} of $K^\psi_{uv}$.}
\end{definition}
\noindent If $K$ is a normal $d$-pseudomanifold and $K^\psi_{uv}$ is obtained from $K$ by an edge folding at $uv$, then 
\begin{equation} \label{edge folding g2}
g_2(K^\psi_{uv}) = g_2(K)+\binom{d}{2}.
\end{equation}
 %Let $\sigma$ be a tetrahedron with the vertex set $\{a,b,c,v\}$. Then the triangle $abc$ is denoted by $\sigma-v$. Moreover, if $\sigma$ is a missing tetrahedron in $K$ then $\sigma-v$ is a missing triangle in $\lk (v,K)$.
Let $vabc$ be a missing tetrahedron in $K$.
 If  $|\lk (v,K)|$ is an orientable surface, then a small neighborhood of $|\partial(abc)|$ in $|\lk (v,K)|$ is an annulus, and if $|\lk (v,K)|$ is a non-orientable surface, then a small neighborhood of $|\partial(abc)|$ in $|\lk (v,K)|$ is either an annulus or a M\"{o}bius strip. 

\begin{lemma}[\cite{BasakSwartz}] \label{lemma:missingtetra2}
Let $K$ be a  normal $3$-pseudomanifold. Let $abcd$ be a missing tetrahedron in $K$ such that $(i)$ for $x\in\{b,c,d\}$, $\partial(K[\{a,b,c,d\}\setminus\{x\}])$ separates  $\lk (x,K),$ and $(ii)$  $\partial(bcd)$  does not separate  $\lk (a,K)$. Then there exists $K'$, a  normal $3$-pseudomanifold such that  $K = (K')^\psi_a$, i.e., $K $ is obtained from a vertex folding at $a \in K'$,  and $abcd$ is the image of the removed facet.  
\end{lemma}

\begin{lemma}[\cite{BasakSwartz}]\label{lemma:missingtetra3}
Let $K$ be a  normal $3$-pseudomanifold. Let $abuv$ be a missing tetrahedron in $K$ such that $(i)$ for $x\in\{a,b\}$, $\partial(K[\{a,b,u,v\}\setminus\{x\}])$ separates  $\lk (x,K),$ and $(ii)$ a small neighborhood of $|\partial(abv)|$ in $|\lk (u,K)|$ is a M\"{o}bius strip. Then a small neighborhood of $|\partial(abu)|$ in $|\lk (v,K)|$ is also a M\"{o}bius strip. Further, there exists $K'$ a  normal $3$-pseudomanifold such that  $K = (K')^\psi_{uv}$, i.e., $K $ is obtained from an edge folding at $uv \in K'$,  and $abuv$ is the removed facet.
\end{lemma}

%%%%%%%%%%%%%%%%%%%%%%%%%%%%%%%%%%%%%%%%%%%%%%%%%%%%%%%%%%%%%%%%SECTION 3

\section{Some lower bounds of $g_2$ for normal 3-pseudomanifolds with one or two singularities}

In this section, we establish a few lower bounds of $g_2$ for a class of  normal 3-pseudomanifolds with one or two singularities. Our general approaches are motivated by the idea used in \cite{Walkup}.

\medskip

\noindent {\em Definition of} $\mathcal{R}$: Let $\mathcal{R}$ be the class of all normal $3$-pseudomanifolds $K$ such that $K$ has one or two singularities and $K$ satisfies the following two properties:

\begin{enumerate}[$(i)$]
\item If $K$ contains the boundary complex of a $3$-simplex as a subcomplex, then $K$ contains the $3$-simplex as well.
	
\item There is no normal $3$-pseudomanifold $K'$ such that $K'$ is obtained from $K$ by a combinatorial operation mentioned in Remark \ref{remark:operation} and $g_2(K')< g_2(K)$.
\end{enumerate}

Now we state a few lemmas, the proofs of which follow from \cite{Walkup} using  Lemma \ref{lemma:homeomorphic}.

\begin{lemma}[Lemma 10.1, \cite{Walkup}]\label{lemma:d>4}
Let $K\in \mathcal{R}$, and $uv$ be an edge in $K$. Then $d(uv)\geq 4$, i.e., $\lk (v,\lk (u))$ has at least four vertices.
\end{lemma}

\begin{lemma}[Lemma 10.2, \cite{Walkup}]\label{lemma:nonempty}
Let $K\in \mathcal{R}$, and $uv$ be an edge in $K$, where $v$ is a non-singular vertex. Then $\lk (u)\cap \lk (v)\setminus\lk (uv)\neq \emptyset$.
\end{lemma}

\begin{lemma} [Lemma 10.4, \cite{Walkup}]\label{lemma:missing-triangle}
Let $K\in \mathcal{R}$, and $u$ be a non-singular vertex in $K$. If $\lk (u)$ contains the boundary complex of a $2$-simplex $\sigma$ as a subcomplex, then $\lk (u)$ contains the $2$-simplex $\sigma$ as well. Thus, for every vertex $v\in \lk (u)$, $ \lk (u)\setminus \{\alpha\in \lk (u)\,:\, v\leq \alpha\}$ does not contain a diagonal edge. %, i.e., an interior edge connecting boundary vertices.
\end{lemma}

\begin{lemma}[Lemma 10.6, \cite{Walkup}]\label{lemma:interior 2}
Let $K\in \mathcal{R}$, and $t$ be a singular vertex in $K$. Let $uv$ be an edge in $K$ such that $uv\not \in \lk (t)$. If $z \in \lk (u)\cap \lk (v)\setminus\lk (uv)$, then $zw \not \in \lk (u)\cap \lk (v)$ for  every non-singular vertex $w\in \lk (uv)$.
\end{lemma}

\begin{lemma}\label{lemma:open edge}
Let $K\in \mathcal{R}$, and $uv$ be an edge in $K$, where $v$ is a non-singular vertex.  Then $\lk (u)\cap \lk (v)\setminus\lk (uv)$ contains some vertices.
\end{lemma}
\begin{proof}
It follows from Lemma \ref{lemma:nonempty} that $\lk (u)\cap \lk (v)\setminus\lk (uv)\neq \emptyset$. If possible, let $\lk (u)\cap \lk (v)\setminus\lk (uv)$ contain an open edge $(z, w)$, where  $z, w\in \lk (uv)$. Then $uvzw$ is a missing tetrahedron in $K$. This contradicts the fact that $K\in \mathcal{R}$. Thus the result follows. %Therefore  $\lk (u)\cap \lk (v)-\lk (uv)$ contains some vertices.
\end{proof} 

By using Lemma \ref{lemma:homeomorphic} in the proof of Lemma 11.1 of \cite{Walkup} we have the following result:

\begin{lemma} \label{lemma:atleast}
	Let $K\in \mathcal{R}$, and  $t$ be a singular vertex in $K$. Let $u$ be a non-singular vertex   in $\lk (t,K)$ such that $\lk (t)\cap \lk (u)\setminus\lk (ut)=(t_1,w]$ or $\{w\}$, where  $w\in \lk (u)\cap \lk (t)\setminus\lk (ut)$ and $t_1\in \lk (ut)$ is a singular vertex. Then $d(tw), d(uw) \geq d(tu)$. 
\end{lemma}

%
%The following result follows from the Euler characteristic of $\lk (t)$ and the definition of normal pseudomanifolds.
%
%\begin{lemma}\label{lemma:f1(st(t))}
%Let $K\in\mathcal{R}$ and $t$ be a vertex of $K$. Let $\chi(\lk (t))$ be the Euler characteristic of $\lk (t)$. Then $f_1(\Star (t))= 4f_0(\Star (t))-3 \chi(\lk (t)) -4$.
%\end{lemma}

\begin{lemma}\label{lemma:degree-singular}
If $K\in \mathcal{R}$, and $K$ contains exactly one singular vertex $t$, then $d(t)\geq 8$.

\end{lemma}
\begin{proof}
The proof follows from the hypothesis of Lemma \ref{lemma:open edge} and the possible triangulations of $\lk (t)$ for $d(t)\leq 7$.
\end{proof}

\begin{lemma}\label{lemma:degree-singular 2}
	Let $K\in \mathcal{R}$, and $t,t_1$ be two singular vertices in $K$. If $d(t)=7$, then $d(tt_1)=6$ and $d(t_1)\geq 8$; otherwise $d(t),d(t_1)\geq 8$.
\end{lemma}
\begin{proof}
If either $d(t)$ or $d(t_1)$ is 6, then it contradicts the hypothesis of  Lemma \ref{lemma:open edge}.
Therefore, $d(t),d(t_1)\geq 7$. If $d(t)=7$, then  $f_1(\lk (t))\geq 18$. Thus $d(ut) = 6$ for some vertex $u\in \lk (t)$. If $u$ is a non-singular vertex, then it contradicts the hypothesis of Lemma \ref{lemma:open edge}. If $d(tt_1)=6$, then $\lk (t)\cap \lk (t_1)\setminus\lk (tt_1)$ is either empty or an open edge. If it is an open edge, then $K$ contains a missing tetrahedron, which contradicts the fact that $K\in \mathcal{R}$. Therefore, $\lk (t)\cap \lk (t_1)\setminus\lk (tt_1)=\emptyset$. If $d(t_1)=7$, then the boundary complex of $\Star (t)\cup \Star (t_1)$ is a triangulated surface with $6$ vertices with the first Betti number more than 1. This is not possible.
Therefore, $d(t_1)\geq 8$.
\end{proof}
	
% For an edge $uv\in K$, we use the notation  $\lk (u)-\Star (v, \lk (u))$ for $\{\sigma \in K\,|\, \sigma \in \lk (u)\mbox{ but } \sigma \notin \lk (v, \lk (u))\}$.
Let $uv$ be an edge in $K$, where $u$ is a non-singular vertex. Define $D_{v}u := \lk (u)\setminus \{\alpha\in \lk (u)\,:\, v \leq \alpha\}$. We say $D_{v}u$ is of type $m(n)$ if $d(u)=m$ and $d(uv)=n$.

\begin{figure}[ht]
\tikzstyle{ver}=[]
\tikzstyle{vertex}=[circle, draw, fill=black!100, inner sep=0pt, minimum width=4pt]
\tikzstyle{edge} = [draw,thick,-]
\centering

\begin{tikzpicture}[scale=0.7]
\begin{scope}[shift={(0,0)}]
\foreach \x/\y/\z in {1/0/a,2.2/1.2/b,-0.2/1.2/c,1/1.2/e,1/2.4/d}{
\node[vertex] (\z) at (\x,\y){};}

\foreach \x/\y in {a/c,a/b,b/d,d/e,d/c,c/e,a/e,b/e}{\path[edge] (\x) -- (\y);}

\node[ver] () at (1,-0.7){6(4)};
\end{scope}

\begin{scope}[shift={(4,0)}]
\foreach \x/\y/\z in {0.6/0/a,1.8/0/b,2.5/1.3/c,1.2/2.4/d,-.2/1.3/e,1.2/1.15/f}{\node[vertex] (\z) at (\x,\y){};}

\foreach \x/\y in {a/e,a/b,b/c,c/d,d/e,a/f,b/f,c/f,d/f,e/f}{\path[edge] (\x) -- (\y);}

\node[ver] () at (1.5,-0.7){7(5)};

\end{scope}

\begin{scope}[shift={(7.5,0)}]
\foreach \x/\y/\z in {2/0/a,0.5/1.5/b,1.5/1.5/c,2.5/1.5/d,3.5/1.5/e,2/3/f}{
\node[vertex] (\z) at (\x,\y){};}

\foreach \x/\y in {a/b,b/f,f/e,e/a,a/c,c/f,f/d,d/a,b/c,c/d,d/e}{\path[edge] (\x) -- (\y);}

\node[ver] () at (2,-0.7){7(4)};
\end{scope}

\end{tikzpicture}
\caption{All possible types of $D_v{u}$, where $u$ is a non-singular vertex in $K$ and $d(u)\leq 7$.}\label{fig:1}
\end{figure}

\begin{lemma}\label{lemma:geq 4}
Let $K\in \mathcal {R}$, and $uv$ be an edge in $K$,  where $u$ is a non-singular vertex. Suppose $\lk (uv)$ contains at most one singular vertex.

\begin{enumerate}[$(i)$]
\item If $d(u)=6$, then $d(v)\geq 9$.

\item If $d(u)=7$ and $D_{v}u$ is of type $7(5)$, then $d(v)\geq 11$.

\item If $d(u)=7$ and $D_{v}u$ is of type $7(4)$, then $d(v)\geq 8$.  
\end{enumerate} 
\end{lemma}
\begin{proof}
$(i)$ Let $d(u)=6$. Then $D_{v}u$ is of type $6(4)$. Suppose $V(\lk (uv))$ = $\{p_1,p_2,p_3,p_4\}$. It follows from Lemma \ref{lemma:open edge} that $\lk (u)\cap \lk (v)\setminus\lk (uv)$  contains exactly one vertex, say $w$. Then $up_i{w}\in K$ for $1\leq i \leq 4$. If some non-singular vertex $p_i\in \lk (uv)\cap \lk (vw)$, then $uvp_i,up_iw,vp_iw\in K$. This implies $\partial(uvw)\subset \lk (p_i)$. 
It follows from Lemma $\ref{lemma:missing-triangle}$ that $uvw\in K$. This is a contradiction.  Therefore, $\lk (uv)\cap \lk (vw)$ does not contain any non-singular vertex. Since $\lk (uv)$ contains at most one singular vertex, $\lk (vw)$ contains at least three more vertices other than the vertices of $\lk (uv)$. Thus, $d(v)\geq 9$.

$(ii)$ Let $d(u)=7$ and $D_{v}u$ be of type $7(5)$. Suppose $V(\lk (uv))=\{p_1,p_2,p_3,p_4,p_5\}$. Since $\lk (u)\cap \lk (v)\setminus\lk (uv)\neq\emptyset$, it  contains exactly one vertex, say $w$. Therefore $up_i{w}\in K$ for $1\leq i\leq 5$. Since $w\in \lk (v)$, it follows from Lemma \ref{lemma:atleast} that $d(vw)\geq 5$. If possible, let  there be a non-singular vertex $p_i\in \lk (uv)\cap \lk (vw)$. Then $\partial(uvw)\subset \lk (p_i)$. It follows from Lemma \ref{lemma:missing-triangle} that $uvw\in K$, which is a contradiction. Therefore, $\lk (uv)\cap \lk (vw)$ does not contain any non-singular vertex. Since $\lk (uv)$ contains at most one singular vertex, $\lk (vw)$ contains at least four more vertices other than the vertices of $\lk (uv)$. Thus, $d(v)\geq 11$.

$(iii)$ Let $d(u)=7$ and $D_{v}u$ be of type 7(4) for some vertex $v\in \lk (u)$. Then Lemma  \ref{lemma:open edge} implies that $\lk (u)\cap \lk (v)\setminus\lk (uv)$ contains  a vertex,  say $w$. It follows from Figure \ref{fig:1} that $d(uv)=4$ and $w$ is connected to 3 vertices of $\lk (uv)$. Let $V(\lk (u))=\{a,b,c,d,e,w,v\}$, $V(\lk (uv))=\{a,b,c,d\}$, and $w$ be connected to the vertices $a,b,c$ in $\lk (uv)$. Since $\lk (uv)$ contains at most one singular vertex, without loss of generality, let $a$ be the singular vertex (if any). It follows from lemma \ref{lemma:missing-triangle} that $b,c\not\in \lk (uw)\cap \lk (vw)$. Therefore, $\lk (vw)$ contains at least $2$ vertices other than the vertices of $\lk (uv)$. Thus, $d(v)\geq 8$. 
\end{proof}

\begin{lemma} \label{lemma:adjacent}
	Let $K\in \mathcal{R}$, and $uv$ be an edge in $K$, where $u$ is a non-singular vertex. 
\begin{enumerate}[$(i)$]	
\item If $d(u)=6$, then $d(v)\geq 8$.
\item  If $d(u)=7$ and $D_v{u}$ is of type $7(5)$, then $d(v)\geq 10$.
\end{enumerate}	
\end{lemma}
	
\begin{proof}
The proof follows by the similar arguments as in the proof of Lemma \ref{lemma:geq 4}.
\end{proof}

\begin{definition}\label{edge-weight}{\rm
Let $K\in \mathcal {R}$. First, we fix a singular vertex $t$ such that $d(t)\geq 8$ (cf. Lemmas \ref{lemma:degree-singular} and \ref{lemma:degree-singular 2}) and $g_2(\lk (t,K)) \geq g_2(\lk (v,K))$ for any other vertex $v$ in $K$. Let $u$ be a vertex in $K$. Then for every vertex $v\in \lk(u,K)$, define
the {\em weight $\lambda(u,v)$} of the vertex $u$ with respect to $v$ as follows:
$$
\begin{array}{lcccl}
	\lambda(u,v) & = & \frac{2}{3} & \mbox{ if } & d(u)=6, \mbox{ and either } u \not \in \Star (t) \mbox{ or } v\not \in \Star (t),\medskip\\
	
	& = & \frac{3}{4} & \mbox{ if } & d(u)=7, d(v,\lk (u))=5, \mbox{ and either } u \not \in \Star (t) \mbox{ or } v\not \in \Star (t),\medskip\\
	
	& = & \frac{1}{2} & \mbox{ if } & d(u)=7, d(v,\lk (u))=4, \mbox{ and either } u \not \in \Star (t) \mbox{ or } v\not \in \Star (t),\medskip\\
	
	& = & \frac{1}{2} & \mbox{ if } & d(u)=8, \mbox{ and either } u \not \in \Star (t) \mbox{ or } v\not \in \Star (t),\medskip\\

	& = & 1-\lambda(v,u) & \mbox{ if } & d(u)\geq 9, d(v)\leq 8, \mbox{ and either } u \not \in \Star (t) \mbox{ or } v\not \in \Star (t),\medskip\\ 

 	& = & \frac{1}{2} & & \mbox{otherwise}. 
	
\end{array}
$$}
\end{definition}

\noindent Then from Lemmas \ref{lemma:degree-singular 2}, \ref{lemma:geq 4} and \ref{lemma:adjacent}, it follows that $\lambda(u,v)+\lambda(v,u)=1$ for every edge $uv$ of $K$. For a vertex $u\in K$, we define the {\em weight} of the vertex $u$ as $\mathcal{W}_u:=\sum_{v \in \lk (u)}\lambda(u,v).$ For a vertex $u\in \lk (t)$, we define the {\em outer weight} of the vertex $u$ as $$\mathcal{O}_u:=\displaystyle{\sum_{\substack{v\in \lk (u)\\ uv\not \in \lk (t)}}}\lambda(u,v).$$

\begin{lemma}\label{lemma:O(t_1)}
Let $K\in \mathcal {R}$, and $t$ be the singular vertex in $K$ as in Definition $\ref{edge-weight}$.  If  $u\in K \setminus \Star (t,K)$, then  $\mathcal{W}_u=\sum_{v \in \lk (u)}\lambda(u,v)\geq 4$.
\end{lemma}
	\begin{proof}
If $d(u) \leq 7$, then it follows from Lemmas \ref{lemma:degree-singular} and  \ref{lemma:degree-singular 2} that $u$ is a non-singular vertex. If $d(u)=6$, then for any vertex $v\in \lk (u)$, $\lambda(u,v)=2/3$. Therefore, $ \sum_{v \in \lk (u)}\lambda(u,v)=6\times 2/3=4.$ If  $d(u)=7$, then for any vertex $v\in \lk (u)$, $D_v{u}$ is of type either $7(5)$ or $7(4)$. It follows from Figure \ref{fig:1} that in both cases, $\lk (u)$ contains five vertices whose $D_v{u}$ is of type $7(4)$ and two vertices for which $D_v{u}$ is of type $7(5)$. Therefore $\sum_{v \in \lk (u)}\lambda(u,v)=5\times 1/2 + 2\times 3/4 = 4$.
	
	 If $d(u)=8$, then for every vertex $v\in \lk (u)$, $\lambda(u,v)=1/2$. Therefore $\sum_{v \in \lk (u)}\lambda(u,v)= 4$. 

If $d(u)=9$,  then it follows from Lemma $\ref{lemma:geq 4}$ that $\lk (u)$ contains no singular vertex $v$ for which $\lambda(u,v)=1/4$ holds. If possible, let $v$ be the other singular vertex such that  $\lambda(u,v)=1/4$. Then $d(v)=7$, $d(uv)=5$ and $t\in \lk (v)$ such that $d(vt)=6$. This implies $t\in \lk (u)$, a contradiction as $u\not\in \Star (t)$. If for every vertex $v\in \lk (u)$, $\lambda(u,v)=1/2$, then we are done. Suppose there is a vertex $v\in \lk (u)$ such that $\lambda(u,v)=1/3$, i.e., $d(v)=6$. Let $\lk (v)=B_{u_1,u_2,u_3,u_4}(u,z)$. Then Lemma \ref{lemma:open edge} implies that $z\in \lk (u)$. It follows from Lemma \ref{lemma:adjacent} that the five vertices $u_1,u_2,u_3,u_4,z \in \lk (u)$ have a degree of at least 8. Further, $d(uz)\geq 4$, and Lemma \ref{lemma:interior 2} suggests that one vertex from the set $\{u_1,u_2,u_3,u_4\}$ must be a singular vertex; otherwise, $d(u)\geq 10$. Let's assume that $u_1$ is the singular vertex. Therefore, $\lk (uz)$ forms a 4-cycle, denoted as $C_4(u_1,z_1,z_2,z_3)$. As there are no two adjacent vertices in $C_4(u_1,z_1,z_2,z_3)$ with a degree of 6, we conclude that at least one of the vertices $z_1,z_2,$ or $z_3$ has a degree greater than or equal to 8. Hence, there are more than five vertices in $\lk (u)$ that contribute a value of $1/2$ to $\lambda$. Consequently, we deduce that $\sum_{v \in \lk (u)}\lambda(u,v)\geq 4$.

If $d(u)=10$, then by the same arguments as above, we have at least five vertices in $ \lk (u)$ that contribute a value of $1/2$ to $\lambda$. Hence,  $\sum_{v \in \lk (u)}\lambda(u,v)\geq 4$.

Finally, consider $d(u)\geq 11$. If there exists a vertex $v\in \lk (u)$ such that $\lambda(u,v)=1/3$ or $1/4$, then $D_{u}v$ must be of the type $6(4)$ or $7(5)$, respectively. Using similar reasoning as above, we conclude that there are at least five vertices in $\lk (u)$ that contribute a value of $1/2$ to $\lambda$. Therefore,  $\sum_{v \in \lk (u)}\lambda(u,v)\geq 4$.
\end{proof}

\begin{lemma}\label{lemma:f1(K)(2)}
Let $K\in \mathcal {R}$, and $t$ be the singular vertex in $K$ as in Definition $\ref{edge-weight}$. Then $$g_2(K) \geq g_2(\lk (t))+ \sum_{u \in \lk (t)}\mathcal{O}_u.$$
\end{lemma}
\begin{proof}
We know that $f_1(\Star (t))=f_1(\lk (t))+f_0(\lk (t))=g_2(\lk (t))+4 f_0(\lk (t))-6=g_2(\lk (t))+4 f_0(\Star (t))-10$.
% Let $f_0(K \setminus \Star (t))$ and $f_1(K \setminus \Star (t))$ denote the number of 0-dimensional  and 1-dimensional simplices of the set $K \setminus \Star (t)$ respectively. 
 It follows from Lemma \ref{lemma:O(t_1)} that $\sum_{u \not \in \Star (t)}\mathcal{W}_u\geq 4f_0(K \setminus \Star (t))$. Thus
	$$
	\begin{array}{lcl}
		f_1(K) & = & f_1(\Star (t)) + f_1(K \setminus \Star (t))\smallskip\\
		
		& = & f_1(\Star (t))+ \displaystyle{\sum_{uv\in (K \setminus \Star (t))}} [\lambda(u,v)+\lambda(v,u)]\smallskip\\
		
%		& = & f_1(\Star (t))+ \displaystyle{ \sum_{u \in K}\sum_{\substack{v\in \lk (u) \\ uv\not \in \Star (t)}}} \lambda(u,v)\smallskip\\
		
		& = & f_1(\Star (t))+ \displaystyle{ \sum_{u \in \lk (t)}\sum_{\substack{v\in \lk (u) \\ uv\not \in \Star (t)}}} \lambda(u,v) + \sum_{u \not \in \Star (t)}\sum_{v\in \lk (u)}\lambda(u,v)\smallskip\\
		
		& = &  g_2(\lk (t))+4 f_0(\Star (t))-10 + \displaystyle{ \sum_{u \in \lk (t)}}\mathcal{O}_u + \sum_{u \not \in \Star (t)}\mathcal{W}_u\smallskip\\
		
%		& \geq & 4f_0(\Star (t))- 3 \chi(\lk (t)) -4 + \displaystyle{ \sum_{u \in \lk (t)}}\mathcal{O}_u + \sum_{u \not \in \Star (t)} 4\smallskip\\
		
		& \geq & g_2(\lk (t))+4 f_0(\Star (t))-10 + \displaystyle{ \sum_{u \in \lk (t)}}\mathcal{O}_u + 4f_0(K \setminus \Star (t))\smallskip\\
		
		& = & 4f_0(K) + g_2(\lk (t))-10 + \displaystyle{ \sum_{u \in \lk (t)}}\mathcal{O}_u. 		
	\end{array}
	$$
Therefore, $g_2(K) \geq g_2(\lk (t))+ \sum_{u \in \lk (t)}\mathcal{O}_u$. This proves the result.
\end{proof}

\begin{lemma} \label{lemma:outside-edge(2)}
Let $K\in \mathcal {R}$, and $t$ be the singular vertex in $K$ as in Definition $\ref{edge-weight}$. Let $u \in \lk (t)$ be a non-singular vertex such that $\lk (u)=B_{u_1,\dots,u_m}(t;z)$. If $u_i$ is a non-singular vertex, for some  $i\in \{1,\dots,m\}$, then $zu_i \not \in \Star (t)$.
\end{lemma}
\begin{proof}
If $zu_i \in \lk (t)$ for some non-singular vertex $u_i$, then $tzu_i \in K$. Further, $zu_iu, u_iut\in K$. Thus, $\partial(tzu)\subset \lk (u_i)$, and by Lemma \ref{lemma:missing-triangle}, $tzu \in K$. This is a contradiction as $z \not \in \lk (tu)$. Therefore $zu_i \not \in \lk (t)$. Since $z,u_i\neq t$, we have $zu_i \not \in \Star (t)$.
\end{proof}

\begin{lemma}\label{lemma:<1(2)}
Let $K\in \mathcal {R}$, and $t$ be the singular vertex in $K$ as in Definition $\ref{edge-weight}$. If $u$ is a non-singular vertex in $\lk (t)$ with $\mathcal{O}_u<1$, then $\lk (u)\cap \lk (t)\setminus\lk (ut)$ contains exactly one vertex, say $z$. Moreover, if  $z$ is non-singular, then $\mathcal{O}_u=0.5$ and  $\lk (u)=B_{u_1,\dots,u_m}(t;z)$.
\end{lemma}

\begin{proof}

It follows from Lemma \ref{lemma:open edge}  that there are some vertices in  $\lk (u) \cap \lk (t) \setminus \lk (ut)$. Since $\mathcal{O}_u<1$, there is only one such vertex, called $z$.

Let $\lk (tu)=C_m(u_1,u_2,\dots,u_m)$, for some $u_1,\dots,u_m\in \lk (t)$. Then by Lemma \ref{lemma:d>4}, $m \geq 4$.  Since $\lk (u)\cap \lk (t)\setminus \lk (tu)$ has only one vertex $z$, $\mathcal{O}_u\geq \lambda(u,z) = 0.5$. If $\mathcal{O}_u=0.5$, then $\lk (u) \setminus \Star (t,\lk (u))$  contains no vertex other than $z$. Therefore, $\lk (u)=B_{u_1,\dots,u_m}(t;z)$. 

If $0.5< \mathcal{O}_u <1$, then $\lk (u) \setminus  \Star (t,\lk (u))$ contains exactly two vertices $z$ and $w$ such that $\lambda(u,z) = 0.5$ and $0<\lambda(u,w)< 0.5$. This implies that $w \not \in \Star (t)$ and $\lambda(u,w)=\frac{1}{4}$ or $\frac{1}{3}$. If $\lambda(u,w)=\frac{1}{4}$, then $\lk (w)=B_{w_1,\dots,w_5}(u;q)$ and if $\lambda(u,w)=\frac{1}{3}$, then $\lk (w)=B_{w_1,\dots,w_4}(u;q)$. Since  $\lk (tu)=C_m(u_1,u_2,\dots,u_m)$ and there are exactly two vertices $z,w\in \lk (u)\setminus \Star (t,\lk (u))$, we have $d(uz)\leq m+1$.

Let $z$ be a non-singular vertex  and $\lambda(u,w)=\frac{1}{4}$ while $\lk (w)=B_{w_{1},w_2,w_{3},w_4,z}(u;q)$.  It follows from Lemma \ref{lemma:open edge} that $q=u_k$ for some $k$. Thus, $u_kw \not \in \lk (u)$, and hence $u_kz\in \lk (u)$. Therefore, $uu_kz, uwz,wu_kz\in K$. Since $\partial(uwu_k)\subset \lk (z)$, by Lemma \ref{lemma:missing-triangle}, $uwu_k \in K$. But $u_k \not \in \lk (uw)$. This is a contradiction. Thus, $\lambda(u,w)\neq\frac{1}{4}$. By the same arguments, we can show that $\lambda(u,w)\neq\frac{1}{3}$. Therefore, $\mathcal{O}_u=0.5$ and $\lk (u)=B_{u_1,\dots,u_m}(t;z)$, $m\geq 4$.
\end{proof}

\begin{lemma}\label{lemma:0.5(2)}
Let $K\in \mathcal {R}$, and $t$ be the singular vertex in $K$ as in Definition $\ref{edge-weight}$.
Let $u \in \lk (t)$ be a non-singular vertex such that $\lk (u)=B_{u_1,\dots,u_m}(t;z)$, where $z$ is a non-singular vertex. If there is a vertex $v \in \lk (u,\lk (t))$ with $\lk (v)=B_{v_1,\dots,v_k}(t;z_1)$, then $z\neq z_1$.
\end{lemma}

\begin{proof}
If $|\lk (z_1)|\not\cong \mathbb{S}^2$, then clearly, $z \neq z_1$.
For $|\lk (z_1)|\cong \mathbb{S}^2$, let $z=z_1$, i.e., $\lk (u)=B_{u_1,\dots,u_m}(t;z)$ and  $\lk (v)=B_{v_1,\dots,v_k}(t;z)$, where $v \in \lk (u,\lk (t))$. Since $uv$ is an edge in $K$, it follows from Lemma \ref{lemma:open edge} that $\lk (u) \cap \lk (v) \setminus \lk (uv)$ contains some vertices. Let $w \in \lk (u) \cap \lk (v) \setminus \lk (uv)$. Then $uvz,uwz,$ and $vwz$ are in $K$. This implies that  $\partial(uvw)  \subset \lk (z)$, but $uvw \not \in K$. This contradicts Lemma \ref{lemma:missing-triangle} and hence $z\neq z_1$.
\end{proof}

\begin{lemma}\label{lemma:O_v(2)}
Let $K\in \mathcal {R}$, and $t$ be the singular vertex in $K$ as in Definition $\ref{edge-weight}$. Then $\sum_{v\in \lk (t)} \mathcal{O}_v \geq f_0(\lk (t)) - 1$. Moreover,  if $\lk (t,K)$ does not contain the other singular vertex, then  $\sum_{v\in \lk (t)} \mathcal{O}_v \geq f_0(\lk (t))$. 
\end{lemma}

\begin{proof}
If for all vertices $v \in \lk (t)$, $\mathcal{O}_v \geq 1$ holds, then trivially  $\sum_{v\in V(\lk (t))} \mathcal{O}_v\geq f_0(\lk (t)) $. We consider the case when some vertices have an outer weight of less than $1$. Let $p_1 \in \lk (t)$ be a non-singular vertex such that $\mathcal{O}_{p_1}<1$. Then by Lemma \ref{lemma:<1(2)}, $\lk (p_1)\cap \lk (t)\setminus\lk (p_1t)$ contains exactly one vertex, say $z_1$. If $z_1$ is non-singular, then  $\mathcal{O}_{p_1}=0.5$  and  $\lk (p_1)=B_{p^1_1,\dots,p^{m_1}_1}(t;z_1)$.

 Let $S_1$ be the set of all non-singular vertices $v\in \lk (t)$ such that $\mathcal{O}_v=0.5$ and  $\lk (v)=B_{v_1,\dots,{v}_{m}}(t;z_1)$, where $z_1$ is the non-singular vertex as above. Then it follows from  Lemma \ref{lemma:0.5(2)} that $p^1_1,\dots,p^{m_1}_1 \not \in S_1$. Let $S_1'= \{p^1_1,p^2_1,p^3_1,p^4_1\}$. 
Then by Lemma \ref{lemma:outside-edge(2)}, $z_1{p^i_1}\not \in \lk (t)$ for at least three $p^i_1$'s of $S_1'$. Therefore,
 
 $$
 \begin{array}{lcl}
 	\sum_{v \in S_1\cup\{z_1\}} \mathcal{O}_v & = & \mathcal{O}_{z_1}  +  \sum_{v \in S_1 }\mathcal{O}_v \medskip\\

 	& \geq &1.5 +  \sum_{v \in S_1} \lambda(z_1,v) + \sum_{v \in S_1} \mathcal{O}_v \medskip\\
 	
 %	& = & 1.5 + \sum_{v \in S_1} (\lambda(z_1,v) +  \mathcal{O}_v) \medskip\\   
 	
 %	& = & 1.5+ \displaystyle{  \sum_{v \in S_1} 1} ~~ (\mbox{since } \lambda(z_1,v) + \mathcal{O}_v= 1) \medskip\\     
 	
 		& = & 1.5+ card(S_1) ~~ (\mbox{since } \lambda(z_1,v) + \mathcal{O}_v= 1).
 	
 \end{array}
 $$
 	Suppose that there is another non-singular vertex $p_2\in V(\lk (t))\setminus S_1$ such that $\mathcal{O}_{p_2}=0.5$ and   $\lk (p_2)=B_{p^1_2,\dots,p_2^{m_2}}(t;z_2)$,  where $z_2\neq z_1$ is also a non-singular vertex. Let $S_2$ be the set of all non-singular vertices $v\in \lk (t)$ such that $\mathcal{O}_v=0.5$ and  $\lk (v)=B_{v_1,\dots,v_m}(t;z_2)$, where $z_2$ is the non-singular vertex as above. Then by Lemma \ref{lemma:0.5(2)}, $p^1_2,\dots,p^{m_2}_2 \not \in S_2$. Let $S_2'= \{p^1_2,p^2_2,p^3_2,p^4_2\}$. 
 	From Lemma \ref{lemma:outside-edge(2)}, ${z_2}p^i_2\not \in \lk (t)$ for at least three $p^i_2$'s in $S_2'$. By similar arguments as above, we have
 	$$\sum_{v \in S_2\cup\{z_2\}} \mathcal{O}_v \geq {1.5+ card(S_2).}$$
Further, by the assumptions on $S_1$ and $S_2$, we have $(S_1\cup\{z_1\})\cap (S_2\cup \{z_2\})=\emptyset$. Therefore, after a finite number of steps, say $n$,  we get a set $\tilde{S}=(S_1\cup\{z_1\})\cup\dots\cup(S_n\cup\{z_n\})$, where $z_1,\dots,z_n$ are non-singular vertices and  $\sum_{v \in \tilde{S}} \mathcal{O}_v \geq  card(\tilde{S})+n/2$.

Let $tt_1$ be an edge in $K$, where $t_1$ is the other singular vertex in $K$. Then $\tilde{S}\subset V(\lk (t)\setminus {t_1})$. Suppose that there is a non-singular vertex $p_3\in V(\lk (t))$ such that $0.5\leq \mathcal{O}_{p_3} <1$ and  $\lk (p_3)\cap \lk (t)\setminus\lk ({p_3}t)$ contains the vertex $t_1$. Let  $P$ be the set of all non-singular vertices $v \in \lk (t)$ such that  $0.5\leq \mathcal{O}_v <1$ and  $\lk (v)\cap \lk (t)\setminus\lk (vt)$ contains only $t_1$. Then,
 
  $$
 \begin{array}{lcl}
 	\sum_{v \in P\cup\{t_1\}} \mathcal{O}_v & = & \mathcal{O}_{t_1}  +   \sum_{v \in P }\mathcal{O}_v \medskip\\

 	& \geq & \sum_{v \in P} \lambda(t_1,v) + \sum_{v \in P} \mathcal{O}_v \medskip\\
 	
 %	& = & \displaystyle{\sum_{v \in P} (\lambda(t_1,v)} +  \mathcal{O}_v) \medskip\\   
 	
% 	&\geq &  \displaystyle{  \sum_{v \in P} 1} ~~ (\mbox{since } \lambda(t_1,v) + \mathcal{O}_v\geq 1) \medskip\\     
 	
 	& \geq &  card(P)  ~~ (\mbox{since } \lambda(t_1,v) + \mathcal{O}_v\geq 1).\medskip\\ 
 	
\end{array}
$$

From our constructions of  $\tilde{S}$ and $P$, it is clear that $\tilde S\cap (P \cup \{t_1\})= \emptyset$. Further, $v \not\in \tilde{S}\cup(P\cup \{t_1\})$ implies $\mathcal{O}_v\geq 1$. Thus,
$$
 \begin{array}{lcl}
  
   \sum_{v \in \lk (t)} \mathcal{O}_v &=& \sum_{v \in \tilde{S}} \mathcal{O}_v +  \sum_{v \in P\cup \{t_1\}} \mathcal{O}_v + \sum_{v \in V(\lk (t)) \setminus (\tilde S\cup (P\cup \{t_1\}))} \mathcal{O}_v\\
 	&\geq & card(\tilde{S}) + n/2 + card(P)  + f_0(\lk (t)) - card(\tilde S\cup (P\cup \{t_1\})) \\
	&=&  card(\tilde{S}) + n/2 +card(P)  + f_0(\lk (t)) -  card(\tilde{S}) - card(P) -1\\
		&=& f_0(\lk (t)) + n/2 -1 \\
    &\geq& f_0(\lk (t)) -1.
	
\end{array}
$$

If $tt_1$ is not an edge in $K$ then $P$ becomes empty, and $v \not\in \tilde{S}$ implies $\mathcal{O}_v\geq 1$. Thus,
$$
 \begin{array}{lcl}
  
   \sum_{v \in \lk (t)} \mathcal{O}_v &=& \sum_{v \in \tilde{S}} \mathcal{O}_v + \sum_{v \in V(\lk (t)) \setminus \tilde S} \mathcal{O}_v\\
 	&\geq & card(\tilde{S}) + n/2 + f_0(\lk (t)) - card(\tilde S) \\
		&\geq& f_0(\lk (t)).
	
\end{array}
$$
This proves the result.
 \end{proof}

\begin{lemma} \label{lemma:bound:0.5(a)}
Let $K\in \mathcal {R}$, and $t$ be the singular vertex in $K$ as in Definition $\ref{edge-weight}$. Let $u\in \lk (t,K)$ be a non-singular vertex such that $\mathcal{O}_u=0.5$ and  $\lk (u,\lk (t))=C_m(u_1,\dots,u_m)$, for some $m\geq 4$. Then $f_0(\lk (t))\geq 2m + 1$. 
\end{lemma}

\begin{proof}
Since $\mathcal{O}_u=0.5$, $\lk (u) \setminus \Star (t,\lk (u))$ contains exactly one vertex, say $z$. Then $\lk (u)\cap \lk (t)\setminus\lk (tu)=\{z\}$ or $(w,z]$ for some singular vertex $w\in \lk (ut)$. It follows from Lemma \ref{lemma:atleast} that $d(tz),d(uz)\geq m$. Then $u_iz\in\lk (u)$, i.e., $uu_i{z}\in K$ for  $1\leq i \leq m$. If $u_i\in \lk (tz)\cap \lk (tu)$ is a non-singular vertex, then $u_itz, u_iuz\in K$. Since $u_itu\in K$, $\partial(utz)\subset \lk (u_i)$. This implies  $u_ituz\in K$ and hence $z\in \lk (tu)$, which is not possible. Thus, $u_iz \not \in\lk (t)$ for each non-singular vertex $u_i$. Therefore, $f_0(\lk (t,K))\geq f_0(\lk (tu))-1+ f_0(\lk (tz))+card\{u\}+card\{z\}=2m+1$.
\end{proof}

\begin{lemma} \label{lemma:4,5}
Let $K\in \mathcal {R}$, and $t$ be the singular vertex in $K$ as in Definition $\ref{edge-weight}$. Let $u\in \lk (t,K)$ be a non-singular vertex such that $4\leq d (ut)\leq 5$. Then either $\mathcal{O}_u=0.5$ or $\mathcal{O}_u\geq 1$.
\end{lemma}

\begin{proof}
It follows from Lemma \ref{lemma:open edge} that $\mathcal{O}_u \geq 0.5$. If possible, let  $0.5<\mathcal{O}_u<1$. Then $\lk (u)\cap \lk (t)\setminus\lk (ut)$ contains exactly one vertex, say $z$, and $\lk (u,K) \setminus \Star (t,\lk (u))$ contains exactly one vertex, say $w\not \in \lk (t,K)$,  other than $z$. Then $d(u)\leq 8$ and $\lambda(u,w)<0.5$. Then $\lambda(w,u)>0.5$, and  Lemma \ref{lemma:geq 4} implies that $d(u)\geq 9$. This is a contradiction.  Thus, the result follows. 
\end{proof}

\begin{lemma} \label{lemma:5,6}
Let $K\in \mathcal {R}$, and $t$ be the singular vertex in $K$ as in Definition $\ref{edge-weight}$ such that $f_0(\lk (t))\leq 10$.  Let $u\in \lk (t,K)$ be a non-singular vertex such that $5\leq d (ut)\leq 6$.  Then $\mathcal{O}_u\geq 1$.
\end{lemma}
\begin{proof}
 If $\lk (u)\cap \lk (t)\setminus\lk (ut)$ contains two or more vertices, then the result follows. Suppose $\lk (u)\cap \lk (t)\setminus\lk (ut)$ contains exactly one vertex, say $z$. Then Lemma \ref{lemma:atleast} implies that $d(uz),d(tz)\geq d(ut)$.
 If $d(ut)=6$, then $\lk (u)$ contains at least two vertices other than $z$ and the vertices of $\lk (ut)$. Thus $\mathcal{O}_u\geq 1$.
   If $d(ut)=5$, then the result follows from Lemmas \ref{lemma:bound:0.5(a)} and \ref{lemma:4,5}.
	\end{proof}
	
	\begin{lemma} \label{lemma:4 cycle}
	Let $K\in \mathcal {R}$, and  $t$ be the singular vertex in $K$ as in Definition $\ref{edge-weight}$ such that $f_0(\lk (t))\leq 10$. Let $u\in \lk (t)$ be a non-singular vertex such that $d(ut)=4$. Then either $\mathcal{O}_u\geq 1$ or $\mathcal{O}_u = 0.5$ and there exists a vertex $z$ in $\lk (t)$ such that $\mathcal{O}_z\geq 2$.
	\end{lemma}	
	
	\begin{proof}
If $\lk (u) \setminus \Star (t,\lk (u))$ contains more than two vertices, then the result follows.
		If $\lk (u) \setminus \Star (t,\lk (u))$ has two vertices, then $D_t{u}$ must be of type $7(4)$, and the outer weight is 0.5 for both the vertices of $\lk (u)\setminus \Star (t,\lk (u))$. Therefore, $\mathcal{O}_u=1$. If $\lk (u) \setminus \Star (t,\lk (u))$ has exactly one vertex, say $z$, then $\mathcal{O}_u=0.5$ and from Lemma \ref{lemma:bound:0.5(a)}, $f_0(\lk (t))\geq 9$. It follows from Lemma \ref{lemma:interior 2} that $\lk (uz)\cap \lk (tz)$ does not contain any non-singular vertex. Thus, $d(tz)=4$ and $\lk (z)\cap \lk (t)\setminus\lk (tz)$ contains at least 4 vertices, and therefore $\mathcal{O}_z\geq 2$.	
\end{proof}

\begin{lemma}\label{lemma:O(u)=1}
Let $K\in \mathcal {R}$, and $t$ be the singular vertex in $K$ as in Definition $\ref{edge-weight}$. Let $u\in \lk (t)$ be a non-singular vertex such that $\mathcal{O}_u =1$ and $d(ut)=n$, where $4\leq n\leq 6$. Then there is a vertex $z\in \lk (t)$ such that $\mathcal{O}_z\geq 1.5$ for $n=4,$ and $\mathcal{O}_z\geq 2$ for $n=5,6$.
\end{lemma}
	
\begin{proof}
It follows from Lemma \ref{lemma:open edge} that  $\lk (u)\cap \lk (t)\setminus\lk (ut)$ contains some vertices. If $\lk (u) \setminus \Star (t,\lk (u))$ contains more than three vertices, then $\mathcal{O}_u>1$, which is a contradiction. If $\lk (u)\setminus \Star (t,\lk (u))$ contains exactly three vertices, then $\mathcal{O}_u=1$ implies that $(i)$ $\lk (u)\cap \lk (t)\setminus lk (ut)$ contains exactly one vertex, say $z$, and $(ii)$ other two vertices, say $p,q\in \lk (u)$ but  $p,q\not \in \lk (t)$, and $\lambda(u,p)=\lambda(u,q)=0.25$. Further, $d(u)\leq 10$. However, $\lambda(u,p)=0.25$, and $\lk (up)$ contains at most one singular vertex implying $d(u)\geq 11$ (cf. Lemma \ref{lemma:geq 4} $(ii)$), which is a contradiction.

Therefore,  $\lk (u) \setminus \Star (t,\lk (u))$ contains exactly two vertices, say $z$ and $w$. It follows from Lemma \ref{lemma:open edge} that $\lk (u)\cap \lk (t)\setminus\lk (ut)$ contains either both $z$ and $w$ or exactly one vertex, say $z$.

\noindent \textbf{Case 1:} Let $\lk (u)\cap \lk (t)\setminus\lk (ut)$ contain only $z$. It follows from Lemma \ref{lemma:atleast} that $d(uz)\geq d(tu)=n$. Therefore, $z$ is connected  with $n-1$ vertices, say $x_1,x_2,\dots,x_{n-1}$, of $\lk (ut)$ in $D_{t}u$. Since $\lk (ut)$ can have at most one singular vertex, at least $n-2$ vertices from $x_1,x_2,\dots,x_{n-1}$ are non-singular. Let $x_1,x_2,\dots,x_{n-2}$ be non-singular vertices. Then the edges $zx_1,zx_2,\dots,zx_{n-2},$ and $zu$ are not in $\lk (t)$. Therefore,   $\mathcal{O} _z\geq \frac{n-1}{2}$.

\smallskip

\noindent \textbf{Case 2:} Let $\lk (u)\cap \lk (t)\setminus\lk (ut)$ contain both $z$ and $w$. For $n=4$, both $z$ and $w$ are connected  with precisely three vertices of $\lk (ut)$ in $D_{t}u$. By the same arguments as in Case 1, we get  $\mathcal{O} _z\geq 1.5$. For $n=5,6$, one vertex, say $z$, is connected  with at least four vertices of $\lk (ut)$ in $D_{t}u$. Therefore, by similar arguments as in Case 1, we get  $\mathcal{O} _z\geq 2$.
\end{proof}

\begin{lemma} \label{lemma:n-2(2)}
Let $K\in \mathcal {R}$, and $t$ be the singular vertex in $K$ as in Definition $\ref{edge-weight}$. Let $u \in \lk (t)$ be a non-singular vertex such that $\lk (ut)$ is a $(n-2)$-cycle and $f_0(\lk (t))=n$. Then $\mathcal{O}_u\geq \ceil[\big]{\frac{n-3}{2}} \times 0.5 + \floor[\big]{\frac{n-3}{2}} \times 0.25 +0.5$.
\end{lemma}

\begin{proof}
It follows from Lemma \ref{lemma:open edge} that  $\lk (u)\cap \lk (t)\setminus\lk (ut)$ contains exactly one vertex, say $z$. Then  Lemma \ref{lemma:atleast} implies that $d(tz), d(uz) \geq d(tu)=n-2$. Therefore, $d(tz)=n-2$. Let $\lk (uz)=C_{m}(z_1,z_2,\dots z_{m})$, for some $m\geq n-2$.  It follows from Lemma \ref{lemma:interior 2} that $\lk (uz)\cap \lk (tu)$ does not contain any non-singular vertices.
Since  $K$ has at most two singularities, $\lk (uz)$ contains at most one singular vertex, say $z_m$ (if it exists), and hence $\mathcal{O}_u\geq \lambda(u,z)+ \sum_{i=1}^ {m-1}\lambda(u,z_i)$. If $\lambda(u,z_i)=0.25$, then $f_0(\lk (z_i))=7$, and hence $f_0(\lk (z_{i-1}))$ and $f_0(\lk (z_{i+1}))$ must be bigger than 8. Therefore $\lambda(u,z_{i-1})=\lambda(u,z_{i+1})=0.5$ (here the summation in subscripts is modulo $m$). Therefore $\mathcal{O}_u\geq \ceil[\big]{\frac{n-3}{2}} \times 0.5 + \floor[\big]{\frac{n-3}{2}} \times 0.25 +0.5$. 
\end{proof}

\begin{lemma} \label{lemma:n-3(2)}
Let $K\in \mathcal {R}$, and $t$ be the singular vertex in $K$ as in Definition $\ref{edge-weight}$. Let $u \in \lk (t)$ be a non-singular vertex such that $\lk (ut)$ is a $(n-3)$-cycle and  $f_0(\lk (t))=n$, $8\leq n \leq 10$. Then  $\mathcal{O}_u \geq 1.33$ for n=8,9 and 
	$\mathcal{O}_u \geq 1.25$ for n=10.
\end{lemma}

\begin{proof} It follows from Lemma \ref{lemma:open edge} that  $\lk (u)\cap \lk (t)\setminus\lk (ut)$ contains at most two vertices.
	
\noindent \textbf{Case 1:} Let $\lk (u)\cap \lk (t)\setminus\lk (tu)=\{z\}$ or $(y,z]$ for some singular vertex $y\in \lk (tu)$. Then by Lemma \ref{lemma:atleast}, $d(tz), d(uz) \geq d(tu)=n-3$. Since $f_0(\lk (ut))=f_0(\lk (t))-3$, $\lk (t) \setminus \Star (u,\lk (t))$ contains exactly two vertices, and one of them is $z$. Therefore, at least $n-4$ vertices of $\lk (ut)$ are joined with $z$ in $\lk (t)$ and at least $n-5$ of them are non-singular. If one of those $n-5$ non-singular vertices is joined with $z$ in $\lk (u)$, then this contradicts the hypothesis of the Lemma \ref{lemma:interior 2}. Therefore, $\lk (u)$ contains at least $n-5$  vertices other than $z$ and the vertices of $\lk (tu)$. Let $z_1,z_2,\dots z_m$ be the vertices i $\lk u$, where $m \geq n-5$. Therefore,  $\mathcal{O}_u\geq \lambda(u,z)+ \sum_{i=1}^m\lambda(u,z_i) \geq 0.5 + \sum_{i=1}^{n-5}\lambda(u,z_i)\geq 1.5,1.75$ for $n=9, 10,$ respectively.
	
	For $n=8$, $f_0(\lk (ut))=5$ and $\lk (u)$ contains at least 3 vertices other than $z$ and the vertices of  $\lk (ut)$. Therefore, $d(u)\geq 10$. If $d(u)=10$, then $\lk (u)$ contains exactly 3 vertices, say $z_1,z_2,$ and $z_3$ other than vertices of  $\lk (ut)$, where $z_1,z_2,z_3\not\in \lk (t)$. Then from Lemma \ref{lemma:geq 4}, $\lambda(u,z_i)\geq 1/3$. Thus, $\mathcal{O}_u\geq1.5$. If $d(u)\geq 11$, then $\lk (u)$ contains at least 4 vertices other than $z$ and the vertices of  $\lk (ut)$, and hence $\mathcal{O}_u\geq 1.5$.

\smallskip

\noindent \textbf{Case 2:}  Let $\lk (u)\cap \lk (t)\setminus\lk (tu)$ contain exactly two vertices, say $z$ and $w$. We claim that $\lk (u) \setminus  \Star (t,\lk (u))$ contains at least three vertices (i.e., one extra vertex other than $z$ and $w$). If possible, let $\lk (u) \setminus  \Star (t,\lk (u))$ have exactly two vertices $z,w$. Since $D_tu$ does not contain any diagonal (cf. Lemma \ref{lemma:missing-triangle}) and $D_tu$ is a triangulated disc, $zw$ must be an edge, and each vertex in $\lk (tu)$ is joined with either $z$ or $w$ (or both) in $\lk (u)$. Let $\lk (zw,\lk (u))=\{p,q\}$. Then $p,q\in \lk (tu)$ and $p, q$ are joined with both $z$ and $w$.

 If possible, let $zw\not\in\lk (t)$. If $p$ (respectively $q$) is non-singular, then it follows from  Lemma \ref{lemma:interior 2} that $p$ (respectively $q$) is not joined with $z$ and $w$ in $\lk (t)$. Further,  Lemma \ref{lemma:interior 2} implies that a non-singular vertex in $\lk (tu)$, which is joined with $z$ (respectively $w$) in $\lk (u)$, is not joined with $z$ (respectively $w$) in $\lk (t)$. Therefore, 	a  non-singular vertex $v (\neq p,q ) \in \lk (tu)$ is joined with at most one of $z$ and $w$ in $\lk (t)$. Since $V(\lk (t))=\{z,w,u\}\cup V(\lk(tu))$ and $\lk (t)$ contains at most one singular vertex, we have $d(tz)+d(tw)\leq n-3$.  If $n\leq 10$, then $d(tz)+d(tw)\leq 7$, which contradicts the hypothesis of Lemma \ref{lemma:d>4}. Therefore, $zw$ must be an edge in $\lk (t)$.
 
 Let $\lk (zw,\lk (t))=\{r,s\}$.  Then $rz, rw, sz,$ and $sw$ are edges in $\lk (t)$. Since $V(\lk (t))=\{z,w,u\}\cup V(\lk(tu))$ and $\lk (t)$ contains at most one singular vertex, without loss of generality, we assume that $r \in  \lk (tu)$ is a non-singular vertex. Since $r \in  \lk (tu)$, $r$ is joined with either $z$ or $w$ in $\lk (u)$. This contradicts the hypothesis of the Lemma \ref{lemma:interior 2}.   
 
 Therefore, $\lk (u) \setminus  \Star (t,\lk (u))$ has at least three vertices.  Let $x$ be the third vertex. If $\lk (u) \setminus  \Star (t,\lk (u))$ has exactly three vertices, then $d(u)=1+n-3+3=n+1$. Therefore, $\lambda(u,x)\geq 0.33$ for $n=8,9$, and $\lambda(u,x)\geq 0.25$ for $n=10$. Thus, $\mathcal{O}_u\geq \lambda(u,z)+\lambda(u,w)+\lambda(u,x)\geq 1.33$ for $n=8,9,$ and 
	$\mathcal{O}_u \geq 1.25$ for $n=10$. If $\lk (u) \setminus  \Star (t,\lk (u))$ has  more than three vertices, then $\mathcal{O}_u \geq 1.5$. 	
\end{proof}

%%%%%%%%%%%%%%%%%%%%%%%%%%%%%%%%%%%%%%%%%%%%%%%%%%%%%%%%%%%%%SECTION 4

\section{Normal 3-pseudomanifolds with exactly one singularity}

In this section we consider  normal 3-pseudomanifolds with exactly one singularity. Let us denote  $\mathcal{R}_1=\{K\in \mathcal{R}\,:\,$  $K$ has exactly one singularity$\}$. Let $K\in \mathcal{R}_1$ and $t$ be the singular vertex in $K$. Then $\lk (t,K)$ is either a connected sum of tori or a connected sum of Klein bottles. In short, we say that $|\lk (t,K)|$ is a closed connected surface with $h$ number of handles, for $h\geq 1$. For $m \geq 4$, let $x_m$ be the number of vertices in $\lk (t,K)$ with degree $m$ in $\lk (t,K)$.

\begin{lemma}\label{lemma:12}
Let $K\in \mathcal {R}_1$, and $t$ be the singular vertex in $K$. Then $\sum_{v \in \lk (t)} \mathcal{O}_v \geq 10$.
\end{lemma}

\begin{proof}
It follows from Lemma \ref{lemma:degree-singular} that $f_0(\lk (t,K))\geq 8$. First, let us assume  $f_0(\lk (t,K))= 8$. It follows from  Lemma \ref{lemma:open edge} that $x_m=0$ for $m=7$. Let $u\in \lk (t,K)$ be any non-singular vertex such that $d(ut)=4$. It follows from Lemmas \ref{lemma:bound:0.5(a)} and \ref{lemma:4,5}  that  $\mathcal{O}_u\geq 1$. Further, Lemmas \ref{lemma:n-2(2)} and \ref {lemma:n-3(2)} imply  $\sum_{v \in \lk (t)} \mathcal{O}_v \geq x_4 + 1.33 x_5 + 2.375 x_6$, where $x_4 + x_5 + x_6  = 8$ and $4 x_4 + 5 x_5 + 6 x_6 = 48$. Thus, by solving the L.P.P., $\sum_{v \in \lk (t)} \mathcal{O}_v \geq 19$.

Now, we assume that  $f_0(\lk (t,K))= 9$. It follows from  Lemma \ref{lemma:nonempty} that $x_m=0$ for $m=8$. Let $u\in \lk (t,K)$ be any non-singular vertex such that $d(ut)=4$. It follows from Lemma \ref{lemma:4,5}  that either $\mathcal{O}_u=0.5$ or $\mathcal{O}_u\geq 1$. If $d(ut)=4$ and $\mathcal{O}_u=0.5$, then by Lemma \ref{lemma:open edge}, $\lk (t)\cap \lk (u)-\lk (tu)$ contains exactly one vertex, say $z$. From Lemma \ref{lemma:d>4}, we get $d(tz), d(uz)\geq 4$. Since $\lk (t,K)$ does not contain any singular vertex, by Lemma \ref{lemma:interior 2}, $\lk (tz,K)\cap \lk (tu,K)=\emptyset$. This implies, $f_0(\lk (t,K))\geq 10$. This is a contradiction. Therefore,   $\mathcal{O}_u\geq 1$. It follows from Lemmas \ref{lemma:5,6}, \ref{lemma:n-2(2)} and \ref {lemma:n-3(2)} that
$\sum_{v \in \lk (t)} \mathcal{O}_v \geq x_4 + x_5 + 1.33 x_6 + 2.75 x_7$, where $x_4 + x_5 + x_6 + x_7 = 9$ and $4 x_4 + 5 x_5 + 6 x_6 +7 x_7 = 54$. Thus, by solving the L.P.P., we get $\sum_{v \in \lk (t)} \mathcal{O}_v \geq 11.97$.
	
If   $f_0(\lk (t,K))\geq 10$, then from Lemma \ref{lemma:O_v(2)}, we have $\sum_{v \in \lk (t)} \mathcal{O}_v\geq  f_0(\lk (t,K)) \geq 10$. This proves the result.
\end{proof}

\begin{remark}{\rm
Let $K\in \mathcal{R}_1$, and $t$ be the singular vertex in $K$. Then the lower bound for  $\sum_{v \in \lk (t)} \mathcal{O}_v$ can  be easily improved from $10$. However, we did not move in that direction, as the lower bound 10 serves all of our purposes. 
}
\end{remark}

\begin{theorem}\label{theorem:6h+2}
If $K\in\mathcal{R}_1$, then  $g_2(K)\geq g_2(\lk (v, K))+10$ for every vertex $v\in K$.
\end{theorem}
\begin{proof}
Let $t$ be the singular vertex in $K$. It follows from Lemma \ref{lemma:12} that $ \sum_{u \in \lk (t)}\mathcal{O}_u\geq 10$.  Since $g_2(\lk (t,K)) \geq g_2(\lk (v,K))$ for any vertex $v$ in $K$, the result follows from  Lemma \ref{lemma:f1(K)(2)}. 
\end{proof}

\begin{theorem}\label{theorem:1-singularity}
Let $K$ be a normal $3$-pseudomanifold with exactly one singularity at $t$ such that $|\lk (t,K)|$ is a connected sum of $n$ copies of tori or Klein bottles for some positive
integer $n$. Then $g_2{(K)} \leq 9+6n$ implies $K$ is obtained from some boundary complexes  of $4$-simplices by a sequence of operations of types connected sums,  bistellar $1$-moves, edge contractions, edge expansions, and vertex foldings. More precisely, the sequence of operations includes exactly $n$ vertex foldings and a finite number of remaining operations. 
\end{theorem}
\begin{proof}
Let $\Delta$ be a normal $3$-pseudomanifold with exactly one singularity at $t$ such that $|\lk (t,\Delta)|$ is a connected sum of $m$ copies of tori or Klein bottles for some positive integer $m$. Then $g_2(\lk (t,K))=6m$. Let  $g_2{(\Delta)} \leq 9+6m$. We have the following observation:

\begin{enumerate}

\item [Observation $1$:] Let $\bar{\Delta}$ (may be $\Delta$ itself) be a  normal 3-pseudomanifold obtained from $\Delta$ by the repeated applications of combinatorial operations mentioned in Remark \ref{remark:operation}, such that there is no normal 3-pseudomanifold $\Delta'$ that is obtained from $\bar \Delta$ by a combinatorial operation mentioned in Remark \ref{remark:operation} and $g_2(\Delta')< g_2(\bar \Delta)$. If $\bar \Delta$ has no missing tetrahedron, then $\bar \Delta\in \mathcal{R}_1$ and hence by Theorem \ref{theorem:6h+2}, $g_2(\bar \Delta)\geq 6m+10$. Thus,  $g_2(\Delta)\geq g_2(\bar \Delta)\geq 6m+10$. This contradicts the given condition. Therefore, $\bar \Delta$ must have a missing tetrahedron.
\end{enumerate}

There can be four types of missing  tetrahedra in $\Delta$:

\begin{enumerate}
\item [Type $1$:] Let $\sigma$ be a missing tetrahedron in $\Delta$ such that $t$ is not a vertex of $\sigma$.

\item [Type $2$:] Let $\sigma$ be a missing tetrahedron in $\Delta$ such that $t\leq \sigma$ and  $\lk (t,\Delta)$ is separated into two portions by the missing triangle formed by the other three vertices of $\sigma$, where one portion is a disc.

\item [Type $3$:] Let $\sigma$ be a missing tetrahedron in $\Delta$ such that $t\leq \sigma$ and  $\lk (t,\Delta)$ is not separated into two portions by the missing triangle formed by the other three vertices of $\sigma$.

\item [Type $4$:] Let $\sigma$ be a missing tetrahedron in $\Delta$ such that $t\leq \sigma$ and  $\lk (t,\Delta)$ is separated into two portions by the missing triangle formed by the other three vertices of $\sigma$, where no portions are triangulated discs.
\end{enumerate}

Now, we are ready to prove our result. Let $K$ be a  normal $3$-pseudomanifold with exactly one singularity at $t$ such that $|\lk (t,K)|$ is a connected sum of $n$ copies of tori or Klein bottles and $g_2{(K)} \leq 9+6n$ for some positive integer $n$.  We use the principle of mathematical induction on $n$. If a normal 3-pseudomanifold $\Delta$ has no singular vertices and  $g_2{(\Delta)} \leq 9$, then we can assume $K=\Delta$, $n=0$, and $t$ is any vertex of $\Delta$. By Proposition \ref{proposition:BG}, we can say that the result is true for $n=0$. Let us assume that the result is true for $0,1,\dots,n-1$, and let $K$ be the  normal $3$-pseudomanifold that corresponds to $n$.

\begin{enumerate}

\item [Step $1$:] Let $\Delta$ be a  normal 3-pseudomanifold obtained from $K$ by repeated applications of the combinatorial operations mentioned in Remark \ref{remark:operation}, such that there is no normal 3-pseudomanifold $K'$ that is obtained from $\Delta$ by a combinatorial operation mentioned in Remark \ref{remark:operation} and $g_2(K')< g_2(\Delta)$. Then by Observation 1, we get $\Delta$ must have a missing tetrahedron.

\item [Step $2$:] Let $\Delta$ have a missing tetrahedron of Type 1 or 2. Then it follows from Lemma \ref{connected sum} that $\Delta$ is formed using a connected sum of $\Delta_1$  and $\Delta_2$. 
 Let $t\in \Delta_1$ (in case of Type 2, we take $t$ as a vertex in $\Delta_1$ such that $\lk (t,\Delta_1)$ is a connected sum of $n$ number of handles). Then $g_2(\Delta_1)\geq g_2(\lk (t,\Delta_1))\geq 6n$.  Therefore $g_2(\Delta_2)=g_2(\Delta)-g_2(\Delta_1)\leq 9$. Thus, after a finite number of steps, we have  $\Delta=\Delta_1\#\Delta_2\#\cdots\#\Delta_n$, where $(i)$ $t\in \Delta_1$ and $|\lk (t,\Delta_1)|$ is a connected sum of $n$ copies of tori or Klein bottles, $(ii)$ $\Delta_1$ has no missing  tetrahedron of Types 1 and 2, $(iii)$ $6n \leq g_2(\Delta_1)\leq 6n+9$, and $(iv)$ for $2\leq i\leq n$, $\Delta_i$ has no singular vertices, and $g_2(\Delta_i)\leq 9$. Thus, by Proposition \ref{proposition:BG} we have, for $2\leq i\leq n$, $\Delta_i$ is a triangulated $3$-sphere and is obtained from some boundary complexes of $4$-simplices by a sequence of operations of types connected sums,  bistellar $1$-moves, edge contractions, and edge expansions. If  there is no normal 3-pseudomanifold $\Delta_1'$, which is obtained from $\Delta_1$ by a combinatorial operation mentioned in Remark \ref{remark:operation} and $g_2(\Delta_1')< g_2(\Delta_1)$, then by Observation 1, $\Delta_1$ has a missing tetrahedron of Type 3 or 4, and we move to Step 3 or 4, respectively. Otherwise, we move to Step 1 and replace $K$ with $\Delta_1$. Since $K$ has a finite number of vertices and $g_2(K)$ is also finite, after a finite number of steps, we must move to either Step 3 or Step 4.
 
\item [Step $3$:]  Let $\Delta_1$ have a missing tetrahedron of Type 3. It follows from  Lemma \ref{lemma:missingtetra2} that $\Delta_1$  is formed using a vertex folding from a  normal 3-pseudomanifold  $\Delta_1'$ at $t\in \Delta_1'$ and  $g_2(\Delta_1')=g_2(\Delta_1)-6$. Here $\lk (t,\Delta_1')$ is a connected sum of $n-1$ copies of tori or Klein bottle and $g_2(\Delta_1')\leq 9+6(n-1)$. Now, the result follows by the induction hypothesis.

\item [Step $4$:] Let $\Delta_1$ have a missing tetrahedron of Type 4. Then it follows from Lemma \ref{connected sum} that $\Delta_1$  is formed using a connected sum of $\Delta_1'$  and $\Delta_1''$. Let $t_1\in \Delta_1'$ and $t_2\in \Delta_1''$ be identified during the connected sum and produce $t \in \Delta_1$.
Let $\lk (t_1,\Delta_1')$ and $\lk (t_2,\Delta_1'')$ be  the connected sum of $n_1$ and $n_2$ copies of tori or Klein bottles, respectively, where $n_1+n_2=n$ for some positive integers $n_1, n_2$. Since $n_1,n_2>0$, both $n_1,n_2<n$. Further, $g_2(\Delta_1')\leq 9+6n_1$ and $g_2(\Delta_1'')\leq 9+6n_2$. Then the result follows by the  induction hypothesis.
\end{enumerate}
From the construction, we can see that the sequence of operations includes exactly $n$ number of vertex foldings and a finite number of remaining operations.
\end{proof}

\begin{remark}\label{remark:sharp1}{\rm
The upper bound in Theorem \ref{theorem:1-singularity} is sharp. In other words, there exists a  normal $3$-pseudomanifold with exactly one singularity such that $g_2{(K)} = 10+6n$ and $K$ is not obtained from some boundary complexes of $4$-simplices by a sequence of operations of types connected sums,  bistellar $1$-moves, edge contractions, edge expansions, and vertex foldings. Let us take a connected sum of a finite number of the boundary complexes of the $4$-simplices, and then apply a handle addition. Let $\Delta_0$ be the resulting 3-dimensional manifold. Then  $g_2{(\Delta_0)} = 10$. Note that we can choose either an orientable or a non-orientable manifold according to our purpose. For $1\leq i \leq n$, let  $\Delta_i$ be a normal 3-pseudomanifold obtained from  a connected sum of a finite number of the boundary complexes of the $4$-simplices by applying a vertex folding at $v_i$, where $v_i\in\Delta_i$. Then  $g_2{(\Delta_i)} = 6$ and  $\lk (v_i,\Delta_i)$ is a torus or Klein bottle (choose the surface according to the purpose).  Let $\Delta$ be the connected sum $\Delta_0\#\Delta_1\#\cdots\#\Delta_n$, where the vertices $v_1,\dots,v_n$ are identified to a single vertex $v$. Then $\Delta$ is a  normal $3$-pseudomanifold with exactly one singularity at $v$ such that $g_2{(\Delta)} = 10+6n$. This  normal $3$-pseudomanifold $\Delta$ will serve  our purpose.
 }
\end{remark}

Let $K$  be formed using a vertex folding from a  normal 3-pseudomanifold  $K'$ at $t\in K'$ and $|K'|$ is a handlebody with its boundary coned off. Then  $K$  is the following pseudomanifold: take $K'[V(K')\setminus \{t\}]$ (the induced subcomplex of $K'$ on the vertex set $V(K')\setminus \{t\}$), identify two triangles (with an admissible bijection between them) on the boundary, then coning off the boundary at $t$. Therefore,  $|K|$ is a handlebody with its boundary coned off.  Further, let $\lk (t_1,\Delta_1)$ and $\lk (t_2,\Delta_2)$ be  the connected sum of $n_1$ and $n_2$ copies of tori or Klein bottles, respectively, where $n_1+n_2=n$ for some positive integers $n_1, n_2$. Let $\Delta_1$ and $\Delta_2$ be triangulated  handlebodies with its boundary coned off at $t_1$ and $t_2$, respectively. Let $\Delta=\Delta_1\#\Delta_2$, where $t_1\in \Delta_1$ and $t_2\in \Delta_2$ are identified to $t \in \Delta$ during the connected sum. Then  $\Delta$ is the following pseudomanifold: take two triangulated handlebodies $\Delta_1[V(\Delta_1)\setminus \{t_1\}]$ and $\Delta_2[V(\Delta_2)\setminus \{t_2\}]$, identify two triangles  from each of  the boundaries, then coning off the boundary at $t$. Therefore,  $|\Delta|$ is a handlebody with its boundary coned off. Therefore, from the proof of Theorem \ref{theorem:1-singularity} and Remark \ref{remark:sharp1}, we have the following result:

\begin{corollary}\label{corollary:handlebody}
	Let $K$ be a normal $3$-pseudomanifold with exactly one singularity at $t$ such that $|\lk (t,K)|$ is a connected sum of $n$ copies of tori or Klein bottles for some positive
integer $n$. Then
	$g_2{(K)} \leq 9+6n$ implies $|K|$ is a handlebody with its boundary coned off. Moreover, there exists a  normal $3$-pseudomanifold with exactly one singularity such that $g_2{(K)} = 10+6n$ and $|K|$ is not a handlebody with its boundary coned off.
\end{corollary}

%%%%%%%%%%%%%%%%%%%%%%%%%%%%%%%%%%%%%%%%%%%%%%%%%%%%%%%%%%%%%%SECTION 5

\section{Normal 3-pseudomanifolds with exactly two singularities}
In this section we will consider normal 3-pseudomanifolds with exactly two singularities. Let us denote  $\mathcal{R}_2=\{K\in \mathcal{R}\, : \, K$ has exactly two singularities\}. Let $K\in \mathcal{R}_2$ and $t$ be the singular vertex in $K$ as in Definition \ref{edge-weight}, i.e., $g_2(\lk (t,K)) \geq g_2(\lk (v,K))$ for any vertex $v$ in $K$. By Lemma \ref{lemma:degree-singular 2}, we can assume that $d(t)\geq 8$. Let $y_4$ be the number of non-singular vertices of degree $4$ in $\lk (t)$  whose outer weight is $0.5$, and let $x_n$ be the number of non-singular vertices of degree $n$ in $\lk (t)$ with outer weight greater than or equal to 1, for $4\leq n \leq 9$.

\begin{lemma}\label{lemma:8(2)}
Let $K\in \mathcal{R}_2$, and $t$ be the singular vertex in $K$ as in Definition $\ref{edge-weight}$. If  $f_0(\lk (t,K))=8$, then $\sum_{v \in \lk (t)} \mathcal{O}_v >10$.
\end{lemma}
\begin{proof}
Since $f_0(\lk (t,K))=8$, it follows from  Lemma \ref{lemma:nonempty} that $x_7=0$. Let $u\in \lk (t,K)$ be any non-singular vertex such that $d(ut)=4$. It follows from Lemmas \ref{lemma:bound:0.5(a)} and \ref{lemma:4,5}  that  $\mathcal{O}_u\geq 1$. We take the outer weight of any seven non-singular vertices of $K$ contained in $\lk (t)$. It follows from  Lemmas \ref{lemma:n-2(2)} and \ref {lemma:n-3(2)} that
$\sum_{v \in \lk (t)} \mathcal{O}_v \geq x_4 + 1.33x_5 + 2.375x_6$. Further, we have $x_4 + x_5 + x_6 = 7$, and $4x_4 + 5x_5 + 6x_6\geq 35$.  Therefore, by solving the L.P.P., we get $\sum_{v \in \lk (t)} \mathcal{O}_v \geq 9.31$. \end{proof}

	\begin{lemma}\label{lemma:9(2)}
		Let $K\in\mathcal{R}_2$, and $t$ be the singular vertex  in $K$ as in Definition $\ref{edge-weight}$. If $f_0(\lk (t,K))=9$, then $\sum_{v \in \lk (t)} \mathcal{O}_v >9$.
	\end{lemma}
	
	\begin{proof}
	Since $f_0(\lk (t,K))=9$, Lemma \ref{lemma:open edge} implies $x_8=0$.  It follows from Lemma \ref{lemma:4,5}  that $\mathcal{O}_u=0.5$ or $\geq 1$, when $d(tu)=4$.  Further, Lemmas \ref{lemma:5,6}, \ref{lemma:n-2(2)}, and \ref{lemma:n-3(2)} imply   $\mathcal{O}_u \geq 1$ when $d(tu)=5$, $\mathcal{O}_u \geq 1.33$ when $d(tu)=6$, and $\mathcal{O}_u \geq 2.75$ when $d(tu)=7$. Let $t_1$ be the singular vertex in $K$ other than $t$. Now we have the following cases:
	
\noindent \textbf{Case 1:} Let there be a non-singular vertex $u\in \lk (t,K)$ such that $\mathcal{O}_u=0.5$. Since $f_0(\lk (t,K))$ $=9$, it follows from Lemma \ref{lemma:4 cycle} that there must be a vertex $z\in \lk (t)$ such that $(i)$ $\mathcal{O}_z\geq 2$, $(ii)$ $d(tu)=d(tz)=4$, and $(iii)$ $\lk (tu)\cap \lk (tz) =\{t_1\}$. Further, if $\lk (t,K)$ contains at least two non-singular vertices, say $u_1$ and $u_2$, of $K$ such that $\mathcal{O}_{u_1}=\mathcal{O}_{u_2}=0.5$, then  we have two vertices $z_1\neq z_2$ such that $\mathcal{O}_{z_1},\mathcal{O}_{z_2}\geq 2$ and $d(tu_1)=d(tu_2)=d(tz_1)=d(tz_2)=4$. Then   $\sum_{v \in \lk (t)} \mathcal{O}_v \geq 0.5y_4 + x_4 + x_5 + 1.33x_6 + 2.75x_7 +2z$, with one of the following conditions:

\begin{enumerate}[$(i)$]
\item  $y_4 + x_4 + x_5 + x_6  + x_7 + z =8$,  $4y_4+4x_4 + 5x_5 + 6x_6 + 7x_7 + 4z \geq 40$, $y_4 =1$, $z= 1$.
\item  $y_4 + x_4 + x_5 + x_6  + x_7 + z =8$,  $4y_4+4x_4 + 5x_5 + 6x_6 + 7x_7 + 4z \geq 40$, $y_4 \geq 2$, $z\geq 2$.
\end{enumerate}
Thus, by solving the L.P.P., we get $\sum_{v \in lk} \mathcal{O}_v \geq 9.16$.

\noindent \textbf{Case 2:} Suppose that for all non-singular vertices $u\in \lk (t,K)$, $\mathcal{O}_u> 0.5$.  It follows from Lemmas \ref{lemma:4,5} and \ref{lemma:5,6} that $\mathcal{O}_u \geq 1$, when $d(tu)=4$ and $5$. Further, Lemmas \ref{lemma:n-2(2)} and \ref{lemma:n-3(2)} imply  $\mathcal{O}_u \geq 1.33$ when $d(tu)=6$, and $\mathcal{O}_u \geq 2.75$ when $d(tu)=7$. 

\noindent \textbf{Case 2a:}  Let there be a vertex $u\in \lk (t)$ such that $d(ut)=5$ and $\mathcal{O}_u=1$. Then we must have a vertex $z$ in $\lk (t)$ such that  $\mathcal{O}_z\geq 2$ and $d(tz)=4$. Then   $\sum_{v \in \lk (t)} \mathcal{O}_v \geq x_4 + x_5 + 1.33x_6 + 2.75x_7 +2z$, where  $x_4 + x_5 + x_6  + x_7 + z \geq 8$,  $4x_4 + 5x_5 + 6x_6 + 7x_7 + 4z \geq 40$, $x_5 \geq 1$, and $z\geq 1$. Thus, by solving the L.P.P., we get $\sum_{v \in lk(t)} \mathcal{O}_v \geq 9.2$.

\noindent \textbf{Case 2b:} 
Suppose that  $\mathcal{O}_u>1$ for all non-singular vertices $u\in \lk (t,K)$, with $d(ut) = 5$, i.e.,   $\mathcal{O}_u>1.16$. Then   $\sum_{v \in \lk (t)} \mathcal{O}_v \geq x_4 + 1.16x_5 + 1.33x_6 + 2.75x_7$, where $x_4 + x_5 + x_6  + x_7 \geq 8$ and  $4x_4 + 5x_5 + 6x_6 + 7x_7 \geq 40$. Thus, by solving the L.P.P., we get $\sum_{v \in lk(t)} \mathcal{O}_v \geq 9.28$.	\end{proof}

\begin{lemma}\label{lemma:10(2)}
	Let $K\in\mathcal{R}_2$, and $t\in K$ be the singular vertex as in Definition $\ref{edge-weight}$.  If $f_0(\lk (t))=10$, then $\sum_{v \in \lk (t)} \mathcal{O}_v >9$.
\end{lemma}
\begin{proof}
 If there is a non-singular vertex $u\in \lk (t)$ such that $\mathcal{O}_u=0.5$, then from Lemma 3.18 we get $\sum_{v \in \lk (t)} \mathcal{O}_v\geq f_0(\lk (t))-1/2$. Therefore,  $\sum_{v \in \lk (t)} \mathcal{O}_v>9$. Now, we assume that, for all non-singular vertex $u\in \lk (t)$, $\mathcal{O}_u>0.5$. Then by Lemmas \ref{lemma:5,6} and \ref{lemma:4 cycle}, we have $\mathcal{O}_u \geq 1$, when $d(tu)=4,5$ and $6$. Further, Lemmas \ref{lemma:n-2(2)} and \ref{lemma:n-3(2)} imply  $\mathcal{O}_u \geq 1.25$ when $d(tu)=7$ and $\mathcal{O}_u \geq 3.125$ when $d(tu)=8$. 
Since $f_0(\lk (t,K))=10$, it follows from  Lemma \ref{lemma:open edge}, $x_9=0$.

\noindent \textbf{Case 1:} If there is a non-singular vertex $u\in \lk (t)$ such that  $\mathcal{O}_u=1$ for $d(tu)=4,5$ or $6$, then By Lemma \ref{lemma:O(u)=1}, there is another vertex $z \in \lk (t)$ such that $\mathcal{O}_z \geq 1.5$. Then $\sum_{v \in lk} \mathcal{O}_v \geq  x_4 + x_5 + x_6 + 1.25x_7 +3.125 x_8 + 1.5z$, where $z\geq 1$ and $x_4 + x_5 + x_6  + x_7 + x_8 + z \geq 9$. Therefore,  $\sum_{v \in lk(t)} \mathcal{O}_v \geq 9.5$.

\noindent \textbf{Case 2:} For all non-singular vertices  $u\in \lk (t)$,  $\mathcal{O}_u>1$. Then $\sum_{v \in lk(t)} \mathcal{O}_v >  x_4 + x_5 + x_6 + x_7 + x_8\geq 9$. 

This proves the result.
\end{proof}

\begin{theorem}\label{theorem:3h}
If $K\in\mathcal{R}_2$, then  $g_2(K)\geq g_2(\lk (v, K))+10$ for every vertex $v\in K$.
\end{theorem}
\begin{proof}
Let $t$ be the singular vertex in $K$ as in Definition \ref{edge-weight}, i.e., $g_2(\lk (t,K)) \geq g_2(\lk (v,K))$ for any vertex $v$ in $K$.  If  $8\leq f_0(\lk (t))\leq 10$, then it follows from Lemmas \ref{lemma:8(2)}, \ref{lemma:9(2)} and \ref{lemma:10(2)} that $ \sum_{u \in \lk (t)}\mathcal{O}_u\geq 10$.  If $f_0(\lk (t))\geq 11$, then it follows from Lemma \ref{lemma:O_v(2)} that $ \sum_{u \in \lk (t)}\mathcal{O}_u\geq 10$. Now, the result follows from  Lemma \ref{lemma:f1(K)(2)}. 
\end{proof}

Let $K$ be a  normal $3$-pseudomanifold with exactly two singularities  at $t$ and $t_1$ such that  $|\lk (t_1)|\cong \mathbb{RP}^2 $. Then $|\lk (t)|$ is a connected sum of $(2m-1)$ copies of $\mathbb{RP}^2$ for some $m\in \mathbb{N}$, and  $g_2(\lk (t, K))=6m-3$. If  $K\in \mathcal{R}_2$, then  $g_2(K)\geq 7+6m$.

\begin{theorem}\label{theorem:2-singularity}
Let $K$ be a  normal $3$-pseudomanifold with exactly two singularities  at $t$ and $t_1$ such that $|\lk (t)|$ is a connected sum of $(2m-1)$ copies of $\mathbb{RP}^2$ and $|\lk (t_1)|\cong \mathbb{RP}^2 $ for some positive integer $m$. Then $g_2{(K)} \leq 6+6m$ implies that $K$ is obtained from some boundary complexes  of $4$-simplices by a sequence of operations of types connected sums,  bistellar $1$-moves, edge contractions, edge expansions, vertex foldings, and edge foldings. More precisely, the sequence of operations includes exactly $(m-1)$ vertex foldings, one edge folding and a finite number of remaining operations. Further, this upper bound is sharp for such normal $3$-pseudomanifolds.
\end{theorem}
\begin{proof}
Let $\Delta$ be a  normal $3$-pseudomanifold with exactly two singularities  at $t$ and $t_1$ such that $|\lk (t,\Delta)|$ is a connected sum of $(2k-1)$ copies of $\mathbb{RP}^2$ and $|\lk (t_1,\Delta)|\cong \mathbb{RP}^2 $ for some positive integer $k$. Then  $g_2(\lk (t, K))=6k-3$. Let $g_2{(\Delta)} \leq 6+6k$.  We have the following observation:

\begin{enumerate}

\item [Observation $1$:] Let $\bar{\Delta}$ (may be $\Delta$ itself) be a  normal 3-pseudomanifold obtained from $\Delta$ by the repeated applications of the combinatorial operations mentioned in Remark \ref{remark:operation}, such that there is no normal 3-pseudomanifold $\Delta'$ which is obtained from $\bar \Delta$ by a combinatorial operation mentioned in Remark \ref{remark:operation} and $g_2(\Delta')< g_2(\bar \Delta)$. If $\bar \Delta$ has no missing tetrahedron, then $\bar \Delta\in \mathcal{R}_2$ and hence by Theorem \ref{theorem:3h}, $g_2(\bar \Delta)\geq  (6k-3)+10=6k+7$. This contradicts the given condition. Therefore, $\bar \Delta$ must have a missing tetrahedron.
\end{enumerate}

There can be five types of missing  tetrahedra in $\Delta$:

\begin{enumerate}
\item [Type $1$:] Let $\sigma$ be a missing tetrahedron in $\Delta$ such that $t$ and $t_1$ are not vertices  of $\sigma$.

\item [Type $2$:] Let $\sigma$ be a missing tetrahedron in $\Delta$ such that $t\leq \sigma$ and $\lk (t,\Delta)$ is separated into two portions by the missing triangle formed by the other three vertices of $\sigma$, where one portion is a disc. If $t_1\leq \sigma$, then $\lk (t_1,\Delta)$ is separated into two portions by the missing triangle formed by the other three vertices of $\sigma$.

\item [Type $3$:] Let $\sigma$ be a missing tetrahedron in $\Delta$ such that $t,t_1\leq \sigma$ and $\lk (t_1,\Delta)$ is not separated into two portions by the missing triangle formed by the other three vertices. Then
a small neighborhood of $|\partial(\Delta[V(\sigma)\setminus \{t_1\}])|$ in $|\lk (t_1,\Delta)|$ is a M\"{o}bius strip, and it follows from Lemma \ref{lemma:missingtetra3}
that  a small neighborhood of $|\partial(\Delta[V(\sigma)\setminus \{t\}])|$ in $|\lk (t,\Delta)|$ is also a M\"{o}bius strip. Further, there exists a  normal $3$-pseudomanifold  $\Delta'$ such that  $\Delta = (\Delta')^\psi_{tt_1}$ is obtained from an edge folding at $tt_1 \in \Delta'$. Therefore, $\Delta'$ has exactly one singularity, say $t$, such that $(i)$ $|\lk (t,\Delta')|$ is a connected sum of $(k-1)$ copies of tori or Klein bottles, and $(ii)$  $g_2(\Delta')\leq 6(k-1)+9$. It follows from Theorem \ref{theorem:1-singularity} that $\Delta'$ is obtained from some boundary complexes of $4$-simplices by a sequence of operations of types connected sums,  bistellar $1$-moves, edge contractions, edge expansions, and vertex foldings.

\item [Type $4$:] Let $\sigma$ be a missing tetrahedron in $\Delta$ such that $t\leq \sigma$ and  $\lk (t,\Delta)$ is not separated into two portions by the missing triangle formed by the other three vertices of $\sigma$.  If $t_1\leq \sigma$, then $\lk (t_1,\Delta)$ is separated into two portions by the missing triangle formed by the other three vertices of $\sigma$.

\item [Type $5$:] Let $\sigma$ be a missing tetrahedron in $\Delta$ such that $t\leq \sigma$ and  $\lk (t,\Delta)$ is separated into two portions by the missing triangle formed by the other three vertices of $\sigma$, where no portions are triangulated discs. If $t_1\leq \sigma$, then $\lk (t_1,\Delta)$ is separated into two portions by the missing triangle formed by the other three vertices of $\sigma$.
\end{enumerate}

Now, we are ready to prove our result. 	Let $K$ be a  normal $3$-pseudomanifold with exactly two singularities  at $t$ and $t_1$ such that $|\lk (t)|$ is a connected sum of $(2m-1)$ copies of $\mathbb{RP}^2$, $|\lk (t_1)|\cong \mathbb{RP}^2 $ and $g_2{(K)} \leq 6+6m$ for some positive integer $m$.

\begin{enumerate}
\item [Step $1$:] Let $\Delta$ be a  normal 3-pseudomanifold obtained from $K$ by repeated applications of the combinatorial operations mentioned in Remark \ref{remark:operation} such that there is no normal 3-pseudomanifold $K'$ which is obtained from $\Delta$ by a combinatorial operation mentioned in Remark \ref{remark:operation} and $g_2(K')< g_2(\Delta)$. Then by Observation 1, $\Delta$ must have a missing tetrahedron.

\item [Step $2$:] Let $\Delta$ have a missing tetrahedron of Type 1 or 2. Then it follows from Lemma \ref{connected sum} that $\Delta$  is formed using a connected sum of $\Delta_1$  and $\Delta_2$. Let $t\in \Delta_1$ (in case of Type 2, we can take $t$ as a vertex in $\Delta_1$ such that $|\lk (t,\Delta_1)|$ is a connected sum of $(2m-1)$ copies of $\mathbb{RP}^2$). Then $g_2(\Delta_1)\geq g_2(\lk (t,\Delta_1))\geq 3(2m-1)$ and $t_1\in \Delta_1$, where $|\lk (t_1)|\cong \mathbb{RP}^2$.  Therefore, $g_2(\Delta_2)=g_2(\Delta)-g_2(\Delta_1)\leq 9$. Thus, after a finite number of steps, we have  $\Delta=\Delta_1\#\Delta_2\#\cdots\#\Delta_n$, where $(i)$ $t,t_1\in \Delta_1$ such that $|\lk (t)|$ is a connected sum of $(2m-1)$ copies of $\mathbb{RP}^2$ and $|\lk (t_1)|\cong \mathbb{RP}^2$ $(ii)$ $\Delta_1$ has no missing tetrahedron of Types 1 and 2, $(iii)$ $6m-3 \leq g_2(\Delta_1)\leq 6m+6$, and $(iv)$ for $2\leq i\leq n$, $\Delta_i$ has no singular vertices and $g_2(\Delta_i)\leq 9$. Thus, by Proposition \ref{proposition:BG} we have, for $2\leq i\leq n$, each $\Delta_i$ is a triangulated $3$-sphere and is obtained from some boundary complexes of $4$-simplices by a sequence of operations of types connected sums,  bistellar $1$-moves, edge contractions, and edge expansions. If  there is no normal 3-pseudomanifold $\Delta_1'$, which is obtained from $\Delta_1$ by a combinatorial operation mentioned in Remark \ref{remark:operation} and $g_2(\Delta_1')< g_2(\Delta_1)$, then by Observation 1, $\Delta_1$ has a missing tetrahedron of Type 3, 4 or 5, and we move to Steps 3, 4 or 5, respectively. Otherwise, we move to Step 1 and replace $K$ with $\Delta_1$. Since $K$ has a finite number of vertices and $g_2(K)$ is also finite, after a finite number of steps we must move to Steps 3, 4 or 5.

\item [Step $3$:]  Let $\Delta_1$ have a missing tetrahedron of Type 3. From the above arguments for Type 3, we get our result.
\end{enumerate}

We use the principle of mathematical induction on $m$. First, we take $m=1$. In this case, a missing tetrahedron will be of Type 1, 2 or 3 only. Therefore, from Step 2 we must move to Step 3 only. Thus, the result is true for $m=1$. Let us assume that the result is true for $1,\dots,m-1$, and let $K$ be the  normal $3$-pseudomanifold that corresponds to $m$. Then we start from Step 1, and  after a finite number of steps we must move to Steps 3, 4 or 5. If we move to Step 3, then we are done. We can use the induction hypothesis if we move to either Step 4 or Step 5.

\begin{enumerate}
\item [Step $4$:]  Let $\Delta_1$ have a missing tetrahedron of Type 4. Then it follows from  Lemma \ref{lemma:missingtetra2} that $\Delta_1$  is formed using a vertex folding from a  normal 3-pseudomanifold  $\Delta_1'$ at $t\in \Delta_1'$ and  $g_2(\Delta_1')=g_2(\Delta_1)-6$. Here $|\lk (t,\Delta_1')|$ is a connected sum of $2m-3$ copies of $\mathbb{RP}^2$ and $g_2(\Delta_1')\leq 6+6(m-1)$. Now, the result follows by the  induction hypothesis.

\item [Step $5$:] Let $\Delta_1$ have a missing tetrahedron $\sigma$ of Type 5. Then it follows from Lemma \ref{connected sum} that $\Delta_1$  is formed using a connected sum of $\Delta_1'$  and $\Delta_1''$. Let $t'\in \Delta_1'$ and $t''\in \Delta_1''$ be identified during the connected sum and produce $t \in \Delta_1$. Without loss of generality, assume  $t_1\in \Delta_1'$ (if $t_1\leq \sigma$, then we can take $t_1$ as a vertex in $\Delta_1'$ such that $|\lk (t_1,\Delta_1')|\cong \mathbb{RP}^2$). Then $|\lk (t',\Delta_1')|$  is a connected sum of $2n_1-1$ copies of $\mathbb{RP}^2$ and $|\lk (t'',\Delta_1'')|$ is  the connected sum of $n_2$ copies of tori or Klein bottles, where $n_1+n_2=m$ for some positive integers $n_1, n_2$. Since $n_1,n_2>0$, both $n_1,n_2<m$. Further, $g_2(\Delta_1')\leq 6+6n_1$ and $g_2(\Delta_1'')\leq 9+6n_2$. Now, the result follows by the induction hypothesis and Theorem \ref{theorem:1-singularity}.
\end{enumerate}

From the construction, we can see that the sequence of operations includes exactly $m-1$ number of vertex foldings, one edge folding, and a finite number of remaining operations. Further, the upper bound in Theorem \ref{theorem:2-singularity} is sharp, i.e., there exists a  normal $3$-pseudomanifold with exactly two singularities $\mathbb{RP}^2$ and $\#_{(2m-1)}\mathbb{RP}^2$  such that $g_2{(K)} = 7+6m$ and $K$ is not obtained from some boundary complexes of $4$-simplices by a sequence of operations of types connected sums,  bistellar $1$-moves, edge contractions, edge expansions, vertex foldings, and edge foldings. We can
construct a normal $3$-pseudomanifold $\Delta$ as in Remark \ref{remark:sharp1}, where  $|\lk (t,\Delta)|$ is  the connected sum of $m-1$ copies of tori or Klein bottles, and then apply an edge folding at some edge $ta$. Then the  normal $3$-pseudomanifold $\Delta^\psi_{ta}$ will serve  our purpose.
\end{proof}

\bigskip

\noindent {\em Proof of Theorem} \ref{main theorem}. 
Let $K$ have exactly one singularity at $t$. Then $|\lk (t,K)|$ is a connected sum of $n$ copies of tori or Klein bottles for some $n\in \mathbb{N}$, and $g_2(\lk (t, K))=6n$.
Since $g_2(K) \leq g_2(\lk (v)) + 9$ for some vertex $v$ in $K$ and $g_2(\lk (v))  \leq g_2(\lk (t)) $, we have $g_2(K) \leq 6n+ 9$.  Therefore, the result follows from Theorem \ref{theorem:1-singularity}. 

Now consider, $K$ has exactly two singularities  at $t$ and $t_1$ such that $|\lk (t_1, K)|\cong \mathbb{RP}^2$. Then $|\lk (t, K)|$ is a connected sum of $(2m-1)$ copies of $\mathbb{RP}^2$ for some $m\in \mathbb{N}$ and  $g_2(\lk (t, K))=3(2m-1)$. Since $g_2(K) \leq g_2(\lk (v)) + 9$ for some vertex $v$ in $K$ and $g_2(\lk (v))  \leq g_2(\lk (t)) $, we have $g_2(K) \leq 6m+6$. Thus, the result follows from Theorem \ref{theorem:2-singularity}.

The sharpness of this bound follows from Remark \ref{remark:sharp1} and Theorem \ref{theorem:2-singularity}. \hfill $\Box$

\medskip

\noindent {\bf Acknowledgement:} The author would like to thank the anonymous referees for many useful comments and suggestions.
The first author is supported by Science and Engineering Research Board (CRG/2021/000859). The second author is supported by CSIR (India). The third author is supported by Prime Minister's Research Fellows (PMRF) Scheme.
  
  {

\end{document}